\documentclass[a4paper,10pt]{amsart}
\usepackage[utf8]{inputenc}
\usepackage[T1]{fontenc}
\usepackage{lmodern}
\usepackage{amsmath}
\usepackage{mathrsfs}
\usepackage{amssymb}
\usepackage{tikz}
\usepackage[all]{xy}
\usetikzlibrary{matrix,arrows}
\usepackage{url}
\usepackage{graphicx}
\usepackage{enumerate}        
\usepackage{amsmath}	
\usepackage{amssymb} 
\usepackage{amsfonts} 
\usepackage{amsthm}
\input xy
\xyoption{all}
\usepackage[T1]{fontenc}
\usepackage[german,english]{babel} 
\usepackage{microtype}
\newcommand{\mcC}{{\mathcal C}}  
\newcommand{\mA}{{\mathcal A}}

\newcommand{\mcL}{{\mathcal L}}

\newcommand{\UU}{{\mathcal U}}
\newcommand{\DD}{{\mathcal D}}

\newcommand{\lieg}{{\mathfrak g}}
\newcommand{\lieu}{{\mathfrak u}}
\newcommand{\NN}{{\mathbb N}}

\newcommand{\Fq}{{\mathbb F}}

\DeclareMathOperator{\Sym}{Sym}

\DeclareMathOperator{\vect}{vect}
\DeclareMathOperator{\abs}{abs}

\DeclareMathOperator{\ad}{ad}
\DeclareMathOperator{\Stab}{Stab}

\DeclareMathOperator{\diag}{diag}
\DeclareMathOperator{\intg}{int}

\DeclareMathOperator{\SU}{SU}
\DeclareMathOperator{\PGL}{PGL}

\DeclareMathOperator{\GL}{GL}
\DeclareMathOperator{\SL}{SL}
\DeclareMathOperator{\PSL}{PSL}
\DeclareMathOperator{\Gal}{Gal}
\DeclareMathOperator{\Aut}{Aut}
\DeclareMathOperator{\Ad}{Ad}
\DeclareMathOperator{\im}{im}

\DeclareMathOperator{\Fix}{Fix}
\DeclareMathOperator{\Hom}{Hom}

\DeclareMathOperator{\Lie}{Lie}

\DeclareMathOperator{\conv}{conv}
\DeclareMathOperator{\chark}{char}

\DeclareMathOperator{\SO}{SO}

\newtheorem{theorem}{Theorem}[section]
\newtheorem{lemma}[theorem]{Lemma}
\newtheorem{proposition}[theorem]{Proposition}
\newtheorem{corollary}[theorem]{Corollary}
\newtheorem{definition}[theorem]{Definition}

\theoremstyle{definition}
\newtheorem{example}[theorem]{Example}

\theoremstyle{remark}
\newtheorem{remark}[theorem]{Remark}%

\usepackage{geometry,float}
\usepackage{bbm,amsmath,amssymb,amsfonts,amsthm}
\usepackage{enumerate}
\usepackage{a4}

\newcommand{\ZZ}{\mathbbm{Z}}    %
\newcommand{\QQ}{\mathbbm{Q}}    %
\newcommand{\RR}{\mathbbm{R}}    %

\newcommand{\GG}{\mathcal{G}}

\renewcommand{\phi}{\varphi}
\renewcommand{\epsilon}{\varepsilon}

\title{The isomorphism problem for almost split Kac--Moody groups}
\author{Guntram Hainke}
\email{guntram.hainke@math.uni-giessen.de}

\begin{document}

\begin{abstract}
In this paper, which is based on a part of the author's Ph.D. thesis \cite{HainkeThesis}, 
we consider the isomorphism problem for almost split Kac--Moody groups, which have been constructed by R\'emy via Galois descent from split Kac--Moody groups as defined by Tits. We show that under certain technical assumptions, any isomorphism between two such groups must preserve the canonical subgroup structure, i.e.\ the twin root datum associated to these groups, which generalizes results of Caprace in the split case. \\
An important technical tool we use is the existence of maximal split subgroups inside almost split Kac--Moody groups, which generalizes the corresponding result of Borel--Tits for reductive algebraic groups. 
\end{abstract}

\maketitle

\section{Introduction}
Let $k$ be an algebraically closed field and $G$ a reductive algebraic group over $k$. Then $G$ is uniquely determined by its Lie algebra (which in turn is determined by a unique classical Cartan matrix $A=(a_{ij})$), the character group of a maximal torus $\Lambda\cong \ZZ^n$, and the set of its roots $c_i \in \Lambda$ and co-roots $h_i \in \Lambda^\vee$ which satisfy $h_i(c_j)=a_{ij}$. \\ 
Conversely, given a datum $\mathcal D$ consisting of a generalized Cartan matrix $A=(a_{ij})$ (which uniquely determines a Kac--Moody algebra $\mathfrak g$), a free $\ZZ$-module $\Lambda$ and elements $c_i \in \Lambda, h_i \in \Lambda \check{}$ which satisfy $h_i(c_j)=a_{ij}$, Tits \cite{Tits2} associates a group functor $\mathcal G_\mathcal D$ on the category of commutative rings. For a field $k$, the value $\mathcal G:=G_\mathcal D(k)$ is called a split Kac--Moody group over $k$. When $A$ is classical, $G$ is a split Chevalley group over $k$, which justifies regarding general Kac--Moody groups as infinite-dimensional Chevalley groups.\\
For an algebraic group $G$ defined over a field $k$, the group of rational points $G(k)$ coincides with the fixed point set of the action of $\Gal(E|k)$ on $G(E)$, where $E|k$ is a Galois extension over which $G$ splits. Using this method of Galois descent, R\'emy \cite{Remy} started the theory of rational points of Kac--Moody groups. In this context, an almost split Kac--Moody group over a field $k$ can be thought of as an infinite-dimensional generalization of the group of $k$-rational points of a $k$-isotropic algebraic group defined over $k$. \\
An important structural feature of a split or almost split Kac--Moody group $G$ is the existence of certain subgroups which form a twin root datum for $G$. In the classical context, this has been used in the fundamental work by Borel--Tits \cite{BorelTits} to describe abstract homomorphisms between isotropic algebraic groups over infinite fields. Caprace \cite{PECAbstract} could provide a similar description for isomorphisms of split Kac--Moody groups.\\
In this paper, we investigate the isomorphism problem for almost split Kac--Moody groups over fields of characteric 0. Our main result is that any such isomorphism must be standard.

\begin{theorem} Let $G,G'$ be two almost split 2-spherical Kac--Moody groups over fields $k,k'$ of characteristic 0. Then any isomorphism $\varphi\colon G \to G'$ is standard, i.e. induces an isomorphism of the corresponding twin root data.  
\end{theorem}

A crucial feature in the proof is to show the existence of a maximal split Kac--Moody subgroup, which gives an inclusion of twin buildings which might be of independent interest. 

\begin{theorem} Let $G$ be an almost split 2-spherical Kac--Moody group over a field $k$ of characteristic 0, let $(H,(U_\alpha)_{\alpha \in \Phi})$ denote its canonical twin root datum and let $S\leq H$ be a maximal split torus. Then there exists a maximal split Kac--Moody group $F \leq G$ such that $F \cap H=S$ and $F \cap U_\alpha=(k,+)$ for all $\alpha \in \Phi$.  
\end{theorem}

The paper is structured as follows. In section 2, we give a general introduction to groups endowed with a twin root datum and their corresponding twin buildings. In the next section, we recall the definition of almost split Kac--Moody groups and investigate the associated twin root datum more closely. In section 4, we investigate maximally split subgroups. 
Finally, the last section deals with the isomorphism problem.\\ 

\textbf{Acknowledgement}. I am grateful to Pierre-Emmanuel Caprace, Bernhard M\"uhlherr and Katrin Tent for useful discussions on the subject of this paper. I would like to thank Brian Conrad for providing a proof of Proposition \ref{existenceofsl2}.

\section{Groups with a twin root datum}
\subsection{Basics}
Let $(W,S)$ be a Coxeter system with associated length function $l$ and $\Phi=\Phi^+ \cup \Phi^-$ the set of its roots. A set of roots $\Psi \subseteq \Phi$ is called {\bf prenilpotent} if there are elements $w,w' \in W$ such that $w\cdot \Psi \subseteq \Phi^+$ and $w'\cdot \Psi \subseteq \Phi^-$. For a prenilpotent pair of roots $\{\alpha, \beta\}$ the {\bf closed root interval} $[\alpha,\beta]$ is defined as
$$[\alpha, \beta]:=\{\gamma \in \Phi: \gamma \supset \alpha \cap \beta \text{ and } -\gamma \supset (-\alpha) \cap (-\beta)\}.$$
Let $[\alpha, \beta):=[\alpha,\beta]\backslash\{\beta\}, 
(\alpha, \beta]:=[\alpha,\beta]\backslash\{\alpha\}$ and
$(\alpha,\beta):=[\alpha,\beta]\backslash\{\alpha,\beta\}.$  

\begin{definition}
 Let $(W,S)$ be a Coxeter system and let $\Phi$ be the set of its roots. Let $G$ be a group and let $(U_\alpha)_{\alpha \in \Phi}$ be a family of non-trivial subgroups. Let $H\leq \cap_{\alpha \in \Phi}N_G(U_\alpha)$ and set $U_+:=\langle U_\alpha: \alpha>0 \rangle$, $U_-:=\langle U_\alpha: \alpha<0\rangle.$ Then $(H,(U_\alpha)_{\alpha \in \Phi})$ is said to be a {\bf twin root datum} for $G$ (of type $(W,S)$) if the following conditions are satisfied:
\begin{enumerate}[(TRD 1)]
 \item $G=H\langle U_\alpha: \alpha \in \Phi \rangle$.
\item For each prenilpotent pair of roots $\{\alpha, \beta\}$, the commutator subgroup $[U_\alpha, U_\beta]$ is contained in 
$U_{(\alpha, \beta)}:=\langle U_\gamma: \gamma \in (\alpha, \beta) \rangle.$
\item For each $s \in S$ and each $u \in U_{\alpha_s} \backslash \{1\},$ there exist $u',u'' \in U_{-\alpha_s}$ such that
$m(u):=u'uu''$ conjugates $U_\beta$ onto $U_{s\beta}$ for all $\beta \in \Phi.$\\
Moreover, for all $u, v \in U_{\alpha_s}\backslash \{1\},m(u)H=m(v)H.$
\item For all $s \in S$, $U_{\alpha_s} \not\subseteq U_-$ and $U_{-\alpha_s}\not\subseteq U_+.$
\end{enumerate}
\end{definition}

\begin{remark} 
This definition is due to Tits \cite{Tits1} who used it to axiomatize features of split Kac--Moody groups. Other examples of groups endowed with a twin root datum (which are also called groups of Kac--Moody type) include $k$-isotropic reductive algebraic $k$-groups, split Chevalley groups and certain ''exotic groups'' \cite{remy2006topological}.
\end{remark}

\begin{definition} \label{standardisom}
Let $G, G'$ be two groups endowed with twin root data $(H,(U_\alpha)_{\alpha \in \Phi(W,S)})$ and $(H', (U'_\beta)_{\beta \in \Phi(W',S')})$. An isomorphism $\varphi\colon G \to G'$ is said to be {\bf standard} (or to {\bf preserve root groups}) if there exists $g \in G'$ such that 
$$ \{\varphi(U_\alpha): \alpha \in \Phi\}=\{gU'_\beta g^{-1}: \beta \in \Phi'\}.$$
\end{definition}

\begin{remark} 
\begin{enumerate}
\item 
If $\varphi$ is standard, it follows that $\varphi(H)=gH'g^{-1}$ and that the corresponding Coxeter groups are isomorphic \cite[Theorem 2.5]{MR2180452}.
\item An important intermediate step when analyzing arbitrary isomorphisms between groups of Kac--Moody type is to show that such an isomorphism is standard. Indeed, in the setting of \cite{BorelTits} it can be first shown that an isomorphism is standard which then allows to conclude that it is ''rational-by-field''. A similar factorization holds for standard isomorphism between two split Kac--Moody groups by \cite{PECAbstract}.  
\end{enumerate}
\end{remark} 

\subsection{Group combinatorics and twin buildings.} Let $G$ be a group endowed with a twin root datum $(H,(U_\alpha)_{\alpha \in \Phi(W,S)})$ of type $(W,S).$ 
For $B_\pm:=HU_\pm$, $N:=H\langle m(u): u \in U_{\alpha_s}\rangle$ and $S$ a set of representatives for the reflections with respect to the simple roots, $(B_+,B_-,N,S)$ is a {\bf twin BN-pair} (see \cite[Definition 6.78]{AbramenkoBrown}) for $G$. \\
In particular, let $\Delta_\pm:=G/B_\pm$ and $\Delta:=(\Delta_+,\Delta_-).$ Then $\Delta$ is a {\bf twin building} of type $(W,S)$ (see \cite[Definition 5.133]{AbramenkoBrown}) and there is a {\bf Bruhat decomposition} of $G$
\index{Bruhat decomposition}
\index{twin!BN-pair}
\index{twin!building}
\index{Birkhoff decomposition}
$$G=\bigcup _{w \in W} B_+ W B_+=\bigcup _{w \in W} B_- W B_-$$
and a {\bf Birkhoff decomposition}
$$G=\bigcup _{w \in W} B_+ W B_-=\bigcup _{w \in W} B_- W B_+.$$

\index{twin!apartment}
A {\bf twin apartment} $\mA=(\mA_+,\mA_-)$ is a subset of $\Delta$ which is isometric to the thin twin building of type $(W,S)$ (see  \cite[Definition 5.171]{AbramenkoBrown}).\\

\index{parabolic subgroup}
A subgroup $P\leq G$ containing a conjugate of $B_\epsilon$ is called a {\bf parabolic subgroup} of sign $\epsilon.$ If $P$ contains $B_\epsilon$, there is a set $J \subseteq S$ such that $P=B_\epsilon W_J B_\epsilon$, where $W_J:=\langle s_i: i \in J \rangle \leq W.$ 
If $W_J$ is finite, $W_J$ (or $J$) is called {\bf spherical.}\\
Let $J \subseteq S.$ Then $L^J:=H \langle U_\alpha: \alpha \in \Phi(W_J,J)\rangle$ is called a {\bf Levi factor}.
\index{Levi factor}

\subsection{Geometric realizations} 
One of the equivalent ways to define a building is to view it as a simplicial complex covered by subcomplexes (the apartments) which are isomorphic to the standard Coxeter complex. We briefly recall two important geometric realizations of this simplicial complex. A very good exposition of the interplay of these two constructions can be found in \cite[Appendix B.4]{KrammerPhD}.\\

{\bf The CAT(0) realization.}
Let $W=\langle (s_i)_{i \in I}: (s_is_j)^{m_{ij}}=1 \rangle$ be a Coxeter group. Let $A:=(-\cos (\frac {\pi} {m_{ij}}))_{i,j}.$ 
Let $V=\bigoplus_{i \in I} \RR e_i$ and let $B_I$ denote the bilinear form induced by $A$, i.e.\ $B_I(e_i,e_j):=a_{ij}.$ Then the representation
$$\rho \colon W \to \GL(V), \rho(s_i)(e_j):=e_j-2 B(e_i,e_j)e_i$$
is called the {\bf standard linear representation} of $W$, which can be shown to be faithful.\\
For a subset $J\subseteq I$, let $V_J:=\bigoplus_{i \in J} \RR e_i$ and write $B_J$ for the restriction of $B_I$ to $V_J.$
For each $J \subseteq S$ such that $W_J$ is spherical let $S_J:=\{x \in V_J: x_i \geq 0, B_J(x,x)=1\}.$ Let $C$ be the intersection of the cone generated by these spherical cells with the half spaces $B_I(e_i,-)\leq 1.$ Then $C$ serves as the model of a chamber. \\
For a building $\Delta$ of type $(W,S)$, this gives a geometric realization of $\Delta$ via the mirror construction (see e.g.\ \cite[Section  4.2.1]{Remy}). Moussong proved that the realization of an apartment in this realization has a natural metric which makes it a CAT(0) space. More precisely, the realization is a CAT(0) polyhedral complex with finitely many shapes of cells. By using retractions, Davis proved that the geometric realization of the entire building is CAT(0).\\
A point in the CAT(0) realization corresponds to a spherical residue of $\Delta.$ If $\Delta=\Delta(G)$ is the building associated to a group $G$ endowed with a BN-pair, then $G$ acts on the CAT(0) realization of $\Delta$ via isometries.\\

{\bf The cone realization.} Again let $(W,S)$ be a Coxeter group and let $\rho\colon W \to \GL(V)$ denote the standard linear representation. A {\bf root} is a vector of the form $a=we_i$ for some $w \in W$ and some standard basis vector $e_i$; let $\Phi=\Phi_+ \cup \Phi_-$ denote the set of all roots. A root $a$ is often identified with the {\bf half-space} 
$$D_a:=\{f \in V^*: f(a)\geq 0\} \subseteq V^*$$ it determines. \index{fundamental chamber}
\index{simple root}
\index{Tits cone} 
\index{half-space}

Let $\overline{C}:=\{f \in V^*: f(e_i) \geq 0 \text{ for all } i \in I\}$ be the so-called {\bf fundamental chamber} and let $\overline{F}_{s_i}:=\{f \in V^*: f(e_i) =0\}$ denote the {\bf wall} associated to the simple root $e_i.$ For an arbitrary root $a$ let $\partial a:=\{f \in V^*: f(a)=0\}$ denote the wall of $a.$\\
Let $W$ act on $V^*$ in the contragredient way, i.e.\ $(w\cdot f)(v):=f(w^{-1}v).$ Then 
$$\mcC:=W\cdot \overline{C}$$ 
is called the {\bf Tits cone} of $W$. It serves as a geometric realization of the Coxeter complex of $W$. \\
Let $\Delta$ be a building of type $(W,S)$, viewed as a discrete set with a $W$-valued metric $\delta$.
Consider the topological space $\Delta_{cone}:=\Delta \times \overline{C}/\sim$, where two points $(c,x),(d,y)$ are identified if and only if $x=y$ and $\delta(c,d) \in W_{J(x)}.$ Here $J(x):=\{s_i \in S: x \in \overline{F}_{s_i}\}$ is the {\bf type} of $x$.\\ 
For a twin building $\Delta=(\Delta_+,\Delta_-)$ the {\bf cone realization} of $\Delta$ is defined as the link of $\Delta_+$ and $\Delta_-$ with the origin of both realizations identified:
$$\Delta_{cone}:=\Delta_+ *\Delta_-/\sim.$$
If $\mA$ is a twin apartment of $\Delta$, it turns out that its geometric realization in $\Delta_{cone}$ is homeomorphic to the realization $\mA'$ of the thin twin building of type $(W,S)$, which can be viewed as two copies of the Tits cone: $\mA'\cong \mathcal C \cup -\mathcal C \subseteq V^*.$ \\
Note that if $W$ is spherical, $\mathcal C=V^*$, while if $W$ is infinite the Tits cone $\mathcal C$ is contained in a half-space. In both cases $\mA=\mathcal C \cup -\mathcal C$ makes good sense.\\
 
Let $\mA$ be a twin apartment of $\Delta$ and let $\Omega \subseteq \Delta$ be a set which is contained in $\mA.$ Identifying $\mA$ with 
$\mathcal{C} \cup -\mathcal{C}$, the {\bf convex hull} of $\Omega$, $\conv_\mA(\Omega)$ is defined as the convex hull of $\Omega$ in $\mA$, and its {\bf vectorial extension}, $\vect_\mA(\Omega)$ as the vector subspace spanned by $\Omega.$
\index{convex hull}
\index{vectorial extension}
\index{generic subspace}
The set $\Omega$ is said to be {\bf generic} if it is, viewed as a subset of $\mathcal{C} \cup -\mathcal{C}$, the intersection of $\mathcal{C} \cup -\mathcal{C}$ with a subspace $L$ of $V^*$ which meets the interior of $\mathcal{C}$: $\Omega=L\cap (\mathcal{C} \cup -\mathcal{C}).$\\
A subset $\Omega \subseteq \Delta_{cone}$ which is contained in a twin apartment $\mA=(\mA_+,\mA_-)$ is called {\bf balanced} if 
$\Omega \cap \mA_+ \neq \emptyset \neq \Omega \cap \mA_-$ and $\Omega$ is contained in the union of a finite number of spherical facets. Here a {\bf spherical facet} $F$ is defined as 
$$F=w\cdot \left(\bigcap_{i \in J} \partial e_i \cap \bigcap_{i \in I \backslash J} D_{e_i}\right)$$ for some $w \in W$ and some spherical subset $J\subseteq I.$ \\
Two points $x,y$ of the cone realization of a twin building are {\bf (geometrically) opposite} if there is a twin apartment $\mA\cong \mathcal C \cup -\mathcal C \subseteq V^*$ containing $x$ and $y$ such that in this identification, $x=-y.$
\index{balanced subset}

\index{facet}
\index{opposite points}

\section{Split and almost split Kac--Moody groups}
In this section we recall the definition of split and almost split Kac--Moody groups and some of their important features. 
\subsection{Kac--Moody algebras}
Let $I$ be a finite index set, $n:=|I|$ and let $A=(a_{ij})_{i,j \in I} \in \ZZ^{n \times n}$ be a {\bf generalized Cartan matrix}, i.e.\ $a_{ii}=2$ for all $i \in I$, 
$a_{ij}\leq 0$ for $i \neq j$ and $a_{ij}=0 \Leftrightarrow a_{ji}=0.$

Let $\Lambda$ be a free $\ZZ$-module of finite rank and denote by $\Lambda^\vee:=\Hom(\Lambda, \ZZ)$ its dual.
For $i \in I$, let $c_i \in \Lambda$ and $h_i \in \Lambda^\vee$  be such that $h_i(c_j)=a_{ij}.$
Then $\DD=(I,A, \Lambda, (c_i)_{i \in I}, (h_i)_{i \in I})$ is called a {\bf Kac--Moody root datum.}\\
The set $\Pi:=\{c_i: i \in I\}$ is called the {\bf base} and the set $\Pi^\vee:=\{h_i: i \in I\}$ the {\bf cobase} of the root datum $\DD$.\\

\index{root datum!simply connected}
\index{root datum!minimal adjoint}
Let $A$ be a generalized Cartan matrix. Two Kac--Moody root data involving $A$ are given by the following two examples.\\
The {\bf simply connected root datum} $\DD^A_{sc}$ associated to $A$ is given by $\Lambda:=\bigoplus_{i \in I} \ZZ e_i$, $c_i:=\sum_{j \in I} a_{ji} e_j$ and $h_i:=e_i^\vee$, where $(e_i^\vee)_{i \in I}$ is the dual basis of $(e_i)_{i \in I}$.\\
The {\bf minimal adjoint root datum} $\DD^A_{min}$ is given by $\Lambda:=\bigoplus_{i \in I} \ZZ e_i$, $c_i:=e_i$ and $h_i:=\sum_{j \in I} a_{ij}e_i^\vee.$ \\
In general, though, neither will the family $(c_i)_{i \in I}$ be free nor generate $\Lambda$. Since for a root datum $\DD=(I,A,\Lambda,(c_i)_{i \in I}, (h_i)_{i \in I})$ its {\bf dual} 
$\DD^t:=(I,A^t, \Lambda^\vee,(h_i)_{i \in I},(c_i)_{i \in I})$ is again a root datum, a similar statement holds for the family $(h_i)_{i \in I}.$\\

\index{Kac--Moody!algebra}
Let $K$ be a field of characteristic 0 and let $\DD$ be a Kac--Moody root datum. The {\bf Kac--Moody algebra} $\lieg=\lieg_\DD$ of type $\DD$ over $K$ is the Lie algebra generated by $\lieg_0:=\Lambda^\vee\otimes_\ZZ K$ and the symbols $e_i,f_i ~ (i=1, \ldots, n)$ subject to the following relations:
$$ [h,e_i] =h(c_i)e_i, ~ [h,f_i]=-h(c_i)f_i ~ \text{ for } h \in \lieg_0, \quad [\lieg_0,\lieg_0]=0,$$
$$ [e_i,f_i]=-h_i \otimes 1, ~ [e_i,f_j]=0 ~ \text{ for } i\neq j,$$ 
$$ (\ad e_i)^{-a_{ij}+1} e_j=(\ad f_i)^{-a_{ij}+1}f_j=0.$$

{\bf The universal enveloping algebra.}
\index{universal enveloping algebra}
\index{universal enveloping algebra!$\ZZ$-form}
Let $\UU_{\lieg_\DD}$ denote the universal enveloping algebra of $\lieg_\DD$. Let $Q:=\ZZ^n$ with standard basis vectors $v_i$. Then there is a well-defined $Q$-grading of $\UU_{\lieg_\DD}$  by setting $\deg h:=0$ for all $h \in \lieg_0$, $\deg e_i:=-\deg f_i:=v_i$ and extending this. This means that there is a family of subspaces $(V_a)_{a \in Q}$ of $\UU_{\lieg_\DD}$ such that 
$\UU_{\lieg_\DD}=\bigoplus_{a\in Q}V_a$ and for $v_a \in V_a, v_b \in V_b$, $[v_a,v_b] \in V_{a+b}.$ As $\lieg_\DD$ can be identified with a subalgebra of $\UU_{\lieg_\DD}$, there is an induced grading $\lieg_\DD=\bigoplus_{a \in Q} \lieg_a.$ If $a$ is such that $\lieg_a\neq 0$, $a$ is called a {\bf root} and $\lieg_a$ a nontrivial {\bf root space}. \index{root} \index{root space}\\

For $u \in \UU_{\lieg_\DD}$ let $u^{[n]}:=\frac{1}{n!}u^n$ and ${u \choose n}:=\frac{1}{n!}u(u-1)\cdots(u-n+1).$\\
Let $\UU_0$ denote the subring of $\UU_{\lieg_\DD}$ generated by all elements ${h \choose n}$, where $h \in \Lambda^\vee$ and $n \in \NN.$ For $i \in \{1, \ldots, n\}$ let $\UU_i$ resp. $\UU_{-i}$ be the subring $\sum_{n \in \NN}\ZZ e_i^{[n]}$ resp. $\sum_{n \in \NN}\ZZ f_i^{[n]}$. Let $\UU_\DD$ be the subring generated by $\UU_0$ and $\UU_i, \UU_{-i} \, (i=1, \ldots, n).$ \\
It can be shown that $\UU_\DD$ is a $\ZZ$-form of $\UU_{\lieg_\DD}$, i.e.\ the canonical map 
$$\UU_\DD \otimes_\ZZ K \to \UU_{\lieg_\DD}$$ is bijective.\\
For a subring $A$ of $\UU_{\lieg_\DD}$ and a ring $R$ let $A_R:=A\otimes_\ZZ R.$ Then $A_R$ inherits a grading. For $M\subseteq (\UU_\DD)_R$, the {\bf support} of $M$ is the set of degrees which appear when decomposing elements of $M$ into their homogeneous components.\\
\index{support}

\index{Weyl group}
{\bf The Weyl group.} 
From the last two sets of defining relations of $\lieg_\DD$ it follows that $\ad e_i, \ad f_i$ are locally nilpotent derivations of $\lieg.$ Then 
$\exp \ad e_i, \exp \ad f_i$ are well-defined automorphisms of $\lieg$. Let 
$$s_i^*:=\exp \ad e_i \cdot \exp \ad f_i \cdot \exp \ad e_i$$
and let $W^*:=\langle s_i^*: i \in I \rangle \leq \Aut(\lieg).$ 

The {\bf Weyl group} of the generalized Cartan matrix $A$ is defined as
$$W:=W_A:=\langle (s_i)_{i \in I} : (s_is_j)^{m_{ij}}=1 \rangle$$
where $m_{ii}:=1$ and for $i \neq j$, $m_{ij}:=2,3,4,6$ or $\infty$ according to whether $a_{ij}a_{ji}=0,1,2,3$ or $\geq 4$.
The group $W_A$ acts on $Q=\ZZ^n$ via $s_i(v_j):=v_j-a_{ij}v_i.$

The connection between $W^*$ and $W$ is as follows: It can be shown that the assignment $s_i^*\mapsto s_i$ extends to a well-defined surjective homomorphism $\pi\colon W^*\to W.$ The action of $W^*$ permutes the root spaces of $\lieg_\DD$, more precisely, we have $w^*\lieg_a=\lieg_{\pi(w^*)a}.$\\
A root $a$ such that $\lieg_a=w^*\lieg_{\pm v_i}$ is called a {\bf real root}. The set of all real roots is denoted by $\Delta^{re}$. It can be identified with the set of roots $\Phi(W,S)$ of the Coxeter group $W$. 

\subsection{The constructive Tits functor}
\index{Tits functor!constructive}
\index{Tits functor}
Let $\DD=(I,A, \Lambda, (c_i)_{i \in I}, (h_i)_{i \in I})$ be a Kac--Moody root datum with associated Weyl group $W=W_A$ and standard generating set $S$. Let
$R$ be a commutative ring with 1.  For $\alpha \in \Phi=\Phi(W,S)$ let $U_\alpha$ denote a group isomorphic to $(R,+)$ and fix an isomorphism $u_\alpha: (R,+) \to U_\alpha$. Let $\widetilde{\mathcal G_\DD}(R)$ denote the free product of 
$T:=\Hom (\Lambda, R^\times)$ and the free product of all $U_\alpha$, $\alpha \in \Phi$. Then the \textbf{constructive Tits functor} $\mathcal G_\mathcal D(R)$ is defined to be a certain quotient of $\widetilde{\mathcal G_\DD}(R)$ such that the canonical images of $(T,(U_\alpha)_{\alpha \in \Phi(W,S)})$ embed in it and form a twin root datum of type $(W,S)$ for $\mathcal G_\DD(R)$ when $R$ is a field. (See \cite{Tits3} for the precise relations.) In this presentation, the torus acts on the simple root groups $U_{\alpha_i}$ via the root $c_i$: $t\cdot u_{\alpha_i}(r)\cdot t^{-1}=u_{\alpha_i}(t(c_i)\cdot r)$, while
two reflections differ by a co-root: $m(u_{\alpha_i}(r))m(u_{\alpha_i}(1))^{-1}=r\cdot h_i$.

\subsection{The adjoint representation}
For each ring $R$, let $\Aut_{filt}(\UU_\DD)_R$ denote the group of $R$-automorphisms of the $R$-algebra $\UU_\DD \otimes_\ZZ R$ which preserve the filtration (or grading) of $(\UU_\DD)_R$ inherited from $\UU_\DD$ and the ideal $\UU_\DD^+\otimes_\ZZ R.$ Here $\UU_\DD^+$ is the ideal of $\UU_\DD$ generated by $\lieg_\DD \UU_{\lieg_\DD} \cap \UU_\DD.$

\begin{theorem}
Let $R$ be a ring. Then there is a homomorphism
$$\Ad\colon \GG_\DD(R) \to \Aut_{filt}(\UU_\DD)_R$$ 
characterized by the conditions
$$\Ad(u_a(r))=\exp (\ad e_a \otimes r)=\sum_{n\geq 0}\frac{(\ad e_a)^n}{n!}\otimes r^n,$$
$$\Ad(T(R)) \text { fixes } (\UU_0)_R \text{ and } \Ad(h)(e_a\otimes r)=h(c_a)\cdot e_a \otimes r$$

for all $h \in T(R)$, $a \in \Phi$ and $r \in R$.
\end{theorem}

\begin{proof}
This is Theorem 9.5.3 in \cite{Remy}. 
\end{proof}

\index{Ad!-locally finite}
\index{Ad!-diagonalizable}

The homomorphism $\Ad$ is called the {\bf adjoint representation}.\\

Let $K$ be a field and let $G:=\GG_\DD(K)$ be a split Kac--Moody group.
A subgroup $H\leq G$ is called {\bf Ad-locally finite} if each $v \in (\UU_\DD)_K$ is contained in a finite-dimensional $\Ad H$-invariant subspace. \\
A subgroup $H \leq G$ is called {\bf Ad-diagonalizable} if there is a basis of $(\UU_\DD)_K$ in which the $H$-action is diagonal.\\
From the explicit description of $\Ad$ it follows that $T(K)$ is Ad-diagonalizable.

\subsection{Group combinatorics}
Let $K$ be a field and let $\GG_\DD$ be a Tits functor. Let $G(K):=\GG_\DD(K).$\\
Let $X^*(T)_{\abs}:=\Hom(T, K^\times)$ denote the group of (abstract) characters of $T$ and $X_*(T)_{\abs}:=\Hom(K^\times,T)$. Then $\Lambda$ injects into $X^*(T)_{\abs}$, while $\Lambda^\vee$ injects into $X_*(T)_{\abs}.$ The group $\Lambda$ is called the {\bf group of algebraic characters} of $T$, while the group $\Lambda^\vee$ is called the {\bf group of algebraic cocharacters} of $T$. 
Let $U_+:=\langle U_\alpha: \alpha>0\rangle$ and $U_-:=\langle U_\alpha: \alpha<0\rangle.$ 
Let $B_+:=TU_+$, $B_-:=TU_-.$ Then $B_+$ (resp. $B_-$) is called the {\bf standard positive} (resp. negative) {\bf Borel subgroup}, while any conjugate of $B_+$ resp. $B_-$ is called a positive resp. negative Borel group. For $\epsilon \in \{\pm 1\}$, the group $U_\epsilon$ is called the {\bf unipotent radical} of $B_\epsilon.$ \\
A positive Borel group $B_1$ and a negative Borel subgroup $B_2$ are called {\bf opposite} if their intersection is a Cartan subgroup.\\
In contrast to the theory of algebraic groups, a (positive or negative) Borel subgroup $B$ of a Kac--Moody group in general is not solvable. Indeed, $B$ is solvable if and only if $W$ is finite.

\subsection {Bounded subgroups}
We recall the close connection between Ad-locally finite groups and fixators of balanced subsets. \\
Let $G$ be a group endowed with a twin root datum $(T,(U_\alpha)_{\alpha \in \Phi(W,S)})$. A subgroup $H\leq G$ is called {\bf bounded} if there exists $n \in \NN$ and $w_1, \ldots, w_n \in W$ such that
$$H \subseteq B_+\{w_1, \ldots, w_n\}B_+ \cap B_-\{w_1, \ldots, w_n\}B_-,$$
i.e.\ for $\epsilon \in \{+,-\},$ $H$ is contained in a finite number of double $B_\epsilon$-cosets.\\
 
The following is Theorem 10.2.2 in \cite{Remy}.
\begin{theorem} Let $G=\GG_\DD(K)$ be a split Kac--Moody group. 
For a subgroup $H\leq G$, the following conditions are equivalent.
\begin{enumerate}
  \item $H$ is bounded.
  \item $H$ fixes a point in the CAT(0)-realization of both $\Delta_+$ and $\Delta_-.$ 
  \item $H$ is Ad-locally finite.
\end{enumerate}
\end{theorem}

Now let $\Omega \subseteq \Delta_{cone}$ be a balanced subset which is contained in the standard apartment $\mA$. By the previous proposition, $H:=\Fix \Omega$ is Ad-locally finite. R\'emy attaches to $H$ a certain finite-dimensional $\Ad H$-invariant subspace whose construction we recall.  \\
Let $\bar{K}$ be an algebraic closure of $K$. 
Let $\mcL_\DD:=\lieg_\DD \cap \UU_\DD$, where $\UU_\DD$ is the $\ZZ$-form of the universal enveloping algebra. Then $\mcL$ has a grading: $\mcL_\DD=\mcL_0 \oplus \bigoplus_{a \in \Phi}\mcL_a.$

Let $\Delta(\Omega):=\{a \in \Phi: \Omega \subseteq a\}$, $\Delta^u(\Omega):=\{a \in \Phi: \Omega \subseteq a, \Omega \subsetneq \partial a\}$ and let $\Delta^m(\Omega):=\{a \in \Phi: \Omega \subseteq \partial a\}.$
Here the roots are viewed as half-spaces in the cone realization.
Write $L:=T\langle U_\alpha: \alpha \in \Delta^m(\Omega)\rangle $ and $U:=\langle U_\alpha: \alpha \in \Delta^u(\Omega)\rangle.$ 
 
\begin{proposition} \label{boundedsubspace}
Let $W=W_\Omega$ be the smallest $Q$-graded subspace of $(\mcL_\DD)_{\bar{K}}$ with the following properties:
\begin{enumerate}
\item $W$ contains $(\mcL_0)_{\bar{K}}$ and $(\mcL_a)_{\bar{K}}$ for all $a \in \Delta(\Omega)$.
\item The $Q$-support of $W$ contains $-\Delta^u(\Omega).$
\item $W$ is stable under $H:=\Fix \Omega.$
\end{enumerate}
Then the following properties hold:
\begin{enumerate}
\item 
$W$ is finite-dimensional and the kernel of $\Ad: H \to \Ad H|_W$ is precisely the center of $H$. 
\item Let $\bar{H}$ (resp. $\bar{T}, \bar{L}, \bar{U}$) denote the Zariski-closure of $\Ad H|_W$ (resp. $\Ad T|_W$, $\Ad L|_W$, $\Ad U|_W$).
Then $\bar{L}$ is a connected reductive $K$-group, $\bar{T}$ is a maximal torus of $\bar{L}$, $\bar{U}$ is unipotent and $\bar{H}=\bar{L} \ltimes \bar{U}$ is a Levi decomposition. \index{Levi decomposition}
\end{enumerate}
\end{proposition}

\begin{proof}
This is \cite[Lemma 10.3.1, Proposition 10.3.6]{Remy}.
\end{proof}
\subsection{Almost split Kac-Moody groups}
We recall R\'emy's construction of almost split Kac--Moody groups, cf. \cite{Remy}, \cite{RemySurvey}. These groups can be obtained via Galois descent, i.e. by taking the fixed points of a certain Galois group action on a split Kac--Moody group. One of the main features of an almost split Kac--Moody group is that it is again endowed with a twin root datum. \\
Let $K$ be a field, $\bar{K}$ an algebraic closure of $K$ and $K_s$ the separable closure of $K$ in $\bar{K}$. Let $\DD$ be a Kac--Moody root datum and let $\GG_\DD$ be a constructive Tits functor. \index{form!prealgebraic}\\
A {\bf prealgebraic $K$-form} of $\GG_\DD$ is a couple $(\GG,\UU_K)$, where $\GG$ is a group functor on the category of field extensions of $K$ which coincides with $\GG_\DD$ over extensions of $\bar{K}$, and $\UU_K$ a $K$-form of the filtered algebra $(\UU_\DD)_{\bar K}$ satisfying
\begin{enumerate}
\item[(PA 1)] The adjoint representation $\Ad$ is Galois-equivariant, i.e. 
for each $K$-algebra $R$ and each $\sigma \in \Gamma:=\Gal(K_s|K)$, the following diagram commutes, where $R_{\bar{K}}:=\bar{K}\otimes_K  R$:
$$\xymatrix{ \GG(R_{\bar{K}}) \ar[r]^-{\Ad}\ar[d]_\sigma & \ar[d]_\sigma  \Aut_{filt}(\UU_\DD)(R_{\bar{K}}) \\
\GG(R_{\bar{K}}) \ar[r]^-{\Ad} &   \Aut_{filt}(\UU_\DD)(R_{\bar{K}}) 
}
$$
\item[(PA 2)] If $\iota\colon K \to L$ is an injection of fields, then $\GG(\iota)\colon\GG(K)\to \GG(L)$ is injective, too.  
\end{enumerate}

Let $E$ be a field satisfying $K \subseteq E \subseteq \bar{K}.$ Then a prealgebraic form $(\GG,\UU_K)$ is said to {\bf split} over $E$ if it is $E$-isomorphic to the split form $(\GG_\DD,(\UU_\DD)_{E})$ over $E$ (see \cite[11.1.5]{Remy} for a precise definition). \\

{\it Convention.} In this subsection, let $(\GG,\UU_K)$ always be a prealgebraic $K$-form of $\GG_\DD$ which is assumed to split over an infinite field $E$ such that $E|K$ is a normal field extension.\\

Let $\Gamma:=\Gal(K^{sep}|K)$ be the absolute Galois group. Then for each field $L\subseteq \bar{K}$ and each $\gamma \in \Gamma$, there is an action of $\Gamma$ on $\GG$ given by $(\gamma\cdot \GG)(L):=\GG(\gamma\cdot L).$ Since $E|K$ is assumed to be normal, $\Gamma$ acts on $\GG(E)$, and since $\GG$ is assumed to split over $E$, each element of $\Gal(K^{\text{sep}}|E)$ acts trivially on $\GG(E)$, i.e. the $\Gamma$-action factors through $\Gal(E|K)$. \\
Fix an isomorphism $\Psi\colon \GG(E) \to \GG_\DD(E).$ By abuse of notation, let 
$T(E)\leq \GG(E)$ again denote the subgroup of $\GG(E)$ which is mapped to the group $T(E)\leq \GG_\DD(E)$. Then $\Gamma$ preserves the conjugacy class of $T(E)$ (cf. \cite[11.2.2]{Remy}). For $\sigma \in \Gamma$, choose $g \in \GG(E)$ such that the so-called {\bf rectification} $\bar{\sigma}:=\text{int }g^{-1} \circ \sigma$ stabilizes $T(E).$ Then $\bar{\sigma}$ induces an automorphism of $W=N(T(E))/T(E)$.\\
\index{rectification}

Let $(\GG,\UU_K)$ be a prealgebraic $K$-form of $G$ which splits over $E$. Then $G$ is said to satisfy (SGR) if for each $\sigma \in \Gamma$, each rectified automorphism $\bar{\sigma}$ of $G(E)$ induces a permutation of the root groups relative to $T(E).$

\begin{remark}By the explicit description of $\Aut(\GG_\DD(E))$ by Caprace (\cite[Theorem A]{PECAbstract}) this condition is empty: $\bar{\sigma}$ automatically preserves root groups. Indeed, by the quoted result any automorphism $\varphi$ can be written as a product $\varphi=\varphi_2 \circ \varphi_1$ of an inner automorphism $\varphi_1$ (which can be chosen to be trivial if $\varphi(T)=T)$ and an automorphism $\varphi_2$ which permutes the root groups: $\varphi_2(x_\alpha(r))=x_{\iota(\alpha)}(c_\alpha \sigma_\alpha(r))$, where $\iota\colon \Phi \to \Phi$ is a bijection, $c_\alpha \in E^\times$ and $\sigma_\alpha \in \Aut(E).$
\end{remark}

It follows that ${\bar \sigma}$ induces a permutation of the roots $\Phi$ of $W$.
Moreover, $\bar{\sigma}$ induces an action on the groups $X^*(T(E))_{\text{abs}}$ resp. $X_*(T(E))_{\text{abs}}$ of abstract characters resp. cocharacters.\\

In this situation, $G=(\GG,\UU)$ is called a {\bf Kac--Moody $K$-group} if for each $\bar{\sigma}$, 
\begin{enumerate}
 \item [(ALG 1)] $\bar{\sigma}$ respects the $Q$-grading of $(\UU_\DD)_E$ and the induced permutation of $Q$ satisfies $\bar{\sigma}(na)=n(\bar{\sigma}(a))$ for all $n \in \NN.$
 \item [(ALG 2)] $\bar{\sigma}$ stabilizes the algebraic characters $\Lambda \leq X^*(T(E))_{\text{abs}}$ resp. the algebraic cocharacters $\Lambda^\vee \leq
X_*(T(E))_{\text{abs}}.$
\end{enumerate}

\index{form!almost split}
\index{form!quasi-split}
\index{Kac--Moody group!almost split}
\index{Kac--Moody group!quasi-split}
Let $G=(\GG,\UU)$ be a Kac--Moody $K$-group. Then $G$ is called {\bf almost split} if the action of $\Gamma$ on $\GG(E)$ stabilizes the conjugacy classes of the standard Borel subgroups $B_+(E)$ and $B_-(E).$
The group $G$ is called {\bf quasi-split} if there are two opposite Borel groups $B_1,B_2$ which are stable under the $\Gamma$-action.\\
Note that a quasi-split Kac--Moody group is automatically almost split. 

\begin{remark} The terminology ``almost split'' stems from the following fact: although an almost split Kac--Moody group has an anisotropic kernel $Z(k)$, 
this group is {\it finite-dimensional}. 
\end{remark}

\index{Galois descent}
{\bf Galois descent.} Let $G=(\GG,\UU)$ be a Kac--Moody $K$-group. Then $G$ is said to be obtained via {\bf Galois descent} if $G$ splits over the separable closure $K_s$ of $K$ in $\bar{K}$ and for each separable field sub-extension $E|K$, the group $\GG(E)$ is precisely the fixed point set of $\Gal(K^{\text{sep}}|E)$ in $\GG(K^\text{sep}).$ In this case, $\GG$ is said to satisfy the condition (DCS).

\subsection{An explicit construction} \label{RemyConstruction}
R\'emy \cite[Ch. 13.2.3]{Remy} gives an explicit construction of quasi-split Kac--Moody groups as follows.  Let $\mathcal G_\DD$ be a constructive Tits functor, $E|K$ a finite Galois extension, $\Gamma:=\Gal(E|K)$ and suppose there is a homomorphism $*\colon \Gamma \to D_A$, where $D_A$ is the Dynkin diagram associated to $A$. Then $*$ gives rise to an action of $\Gamma$ on $\GG_\DD(E)$, and the set $G(k)$ of $\Gamma$-fixed points is a quasi-split Kac--Moody group.
\begin{example} 
\label{SU3example}
The following example is given in \cite[3.5.B]{RemySurvey}. Let $E|K$ be a separable quadratic field extension, $\Gal(E|K)=\langle \sigma \rangle$ and let $\GG_\DD$ be the affine Kac--Moody group $\GG_\DD(K)=\SL_3(K[t,t^{-1}])$. Let $\SU_3(K)\leq \SL_3(E)$ denote the group of matrices which preserve
a fixed three-dimensional $\sigma$-Hermitian form of Witt index 1. Then the group $\SU_3(K[t,t^{-1}])$ is a quasi-split Kac--Moody group obtained by the $*$-action where $\sigma^*$ switches two nodes of the diagram associated to $\GG_\DD$.
\end{example}
More generally, there is the following class of examples of affine quasi-split Kac--Moody groups.
\begin{proposition} Let $\GG$ be a connected simply connected almost simple algebraic group defined over $\Fq_q$ which is $\Fq_q$-isotropic. 
Then for any field $K$ containing $\Fq_q,$ the group
$\GG(K[t,t^{-1}])$ is an almost split Kac--Moody $\Fq_q$-group.
\end{proposition}
\begin{proof} This follows from \cite[Chapter 11]{Remy}. A detailed proof is given in \cite[Proposition 10.2]{BuxGramlich}. 
\end{proof}

\subsection {The Galois action on the building} 
Let $K$ be a field, let $E|K$ be a normal field extension, where $E$ is infinite, and let $\Gamma:=\Gal(E|K)$. Let $G$ denote an almost split Kac--Moody $K$-group obtained by Galois descent which splits over $E$.\\
Let $\Delta=(G(E)/B_+(E),G(E)/B_-(E))$ denote the twin building associated to the group $G(E) \cong \GG_\DD(E).$ The $\Gamma$-action on $G(E)$ then gives rise to an action on $\Delta$ since it preserves the respective conjugacy classes of $B_+,B_-$, cf. \cite[11.3.2]{Remy}.

\index{rectification}
Moreover, there is a better rectification of automorphisms available, that is, for each $\sigma \in G$ there is a $g_\sigma \in G(E)$ (well-defined up to an element in $T(E)$) such that $\sigma^*:=\text{int }g_\sigma^{-1} \circ \sigma$ stabilizes both $B_+(E)$ and $B_-(E).$

\index{action!*-}
This gives a well-defined action of $\Gamma$ on $W$, called the $*$-{\bf action}. This action stabilizes the generating set $S$, i.e. the action is by diagram automorphisms (\cite[11.3.2]{Remy}).\\
It follows that $\Gamma$ acts on the CAT(0)-realization of the buildings $\Delta_+,\Delta_-$. Although $\Gamma$ might be infinite (there is no assumption that $E|K$ is finite, i.e. that $G$ splits over a finite extension of $K$), it can be shown that each orbit is bounded \cite[11.3.4]{Remy}, so by the Bruhat-Tits fixed point theorem, there are fixed points in both halves of the twin building. 
By the dictionary relating the building to its CAT(0)-realization, this is equivalent to saying that there are spherical residues $R_+,R_-$ in both buildings which are stable under the Galois group. The residues $R_+,R_-$ in general will not be chambers, though. Indeed, $\Gamma$ will fix two opposite chambers if and only if $G$ is quasi-split. \\

{\bf The action on the cone realization.} Similarly, $\Gamma$ acts on the cone realization $\Delta_{cone}$ of $\Delta$. Let $\Delta_{cone}^\Gamma$ denote the set of fixed points, then it is clear that $G(K)$ acts on $\Delta_{cone}^\Gamma.$ In what follows, certain subsets of $\Delta_{cone}^\Gamma$ will be singled out, the stabilizers of which then will form the ingredients of a twin root datum for $G(K).$\\

\index{relative!apartment}
To start with, a maximal generic subspace (i.e.\ a sub-vectorspace of an apartment which meets the interior of the Tits cone) which is fixed by $\Gamma$ is called a {\bf $K$-apartment.} These can be shown to exist if $G$ splits over the separable closure of $K$. \\
In the cone realization of the standard twin apartment, such a generic subspace $L$ is given by
$$L=\{x \in V^*: e_i(x)=0 \; \forall \;i: s_i \in S_0 \text { and } e_i(x)=e_j(x) \text{ for }\Gamma^*s_i=\Gamma^*s_j\},$$ cf. \cite[Lemma 12.6.1]{Remy}.
Here $S_0$ is the type of the facet containing a maximal $K$-chamber $F$, see below. Note that the type of a chamber is $\emptyset$.

A {\bf $K$-facet} is the set of $\Gamma$-fixed points of a $\Gamma$-stable facet. A maximal $K$-facet is a {\bf $K$-chamber}. A {\bf $K$-root} (resp. {\bf $K$-half-apartment}, resp. {\bf $K$-wall}, resp. {\bf $K$-panel}) is an apartment (resp. half-apartment, resp. wall, resp. panel) relative to a $K$-apartment $A_K$, i.e. the trace of the corresponding object on $A_K$, which is assumed to be non-empty.\\
Two $K$-chambers of the same sign are called {\bf adjacent} if they contain a common $K$-panel in their closure.\\
Two $K$-chambers of opposite sign are called (geometrically) {\bf opposite} if there is a twin apartment which contains them and in which they are opposite.\\

For a given $K$-apartment $A_K$, $\Delta_K^{re}(A_K)$ is defined as the set of all {\it real} $K$-roots, i.e. those whose relative wall is again a generic subspace, and $\Phi_K(A_K)$ as the set of all $K$-half-apartments relative to $A_K$.\\
For a real $K$-root $a$, let $a^\natural$ denote its restriction to $A_K.$
Then let 
$$\Delta_a:=\{b \in \Delta^{re}: \exists \lambda \geq 1: b^\natural=\lambda a^\natural\}.$$
Note that $\Delta_a$ is a prenilpotent set of roots which is $\Gamma$-stable. \\

\index{standardisation!}
\index{standardisation!rational}
\index{standardisation!compatible}

Finally, a {\bf standardisation} of the cone realization $\Delta_{cone}$ of $G(E)$ is a triple $(A,C,-C)$ where $A$ is a twin apartment which contains the two opposite chambers $C$ and $-C$ (this corresponds to fixing a maximal torus $T$ and two opposite Borel groups $B_1,B_2$ such that $B_1 \cap B_2=T$). A {\bf rational standardisation} is a triple $(A_K,F,-F)$ where $A_K$ is a $K$-apartment and $F,-F$ are two opposite $K$-chambers which are contained in $A_K.$ Two of these triples are called {\bf compatible} if $A$ contains $A_K$ and $C,-C$ contain $F,-F$ respectively.

\subsection{The twin root datum of an almost split group}
Let $K \subseteq E \subseteq K^\text{sep}$ be an inclusion of fields and let $G$ be an almost split Kac--Moody $K$-group which is obtained by Galois descent and splits over $E$. 
For a subgroup $U\leq G(K^{sep})$, let $U(E):=G(E) \cap U$ denote the {\bf group of $E$-rational points} of $U$.\\
A $\Gamma$-invariant parabolic subgroup $P$ of $G$ is called a {\bf $K$-parabolic subgroup}.
Such a $K$-parabolic group is precisely the stabilizer of a $K$-facet.\\
{\bf The anisotropic kernel.} Let $(A_K,F,-F)$ be a rational standardisation. Then $Z:=Z(A_K):=\Fix_{G(K^{sep})}(A_K)$ is called the {\bf anisotropic kernel} (with respect to $A_K$). Let $Z(K)$ denote the set of its $K$-rational points.\\
Let $\Omega:=F \cup -F$. Then $\Ad_\Omega(Z(A_K))$ is isomorphic to a semisimple algebraic $K$-group which is $K$-anisotropic. It follows that $Z$ contains a {\bf maximal $K$-split torus} $T_d(K)$, which can be identified with the connected component of the identity of its center (\cite[12.5.2]{Remy}). The set of all $G(K)$-conjugates of $T_d(K)$ is in bijection with the $K$-apartments.\\
{\bf Rational root groups.} For a real $K$-root $a$, let $V_a:=\langle U_b: b \in \Delta_a \rangle (K)$. By (DCS), $V_a$ is just the fixed point group of $\Gamma$ acting on the $\Gamma$-invariant group 
$U_{a^\natural}:=\langle  U_b: b \in \Delta_a \rangle$.

{\bf Rank 1 groups.} Let $E$ be a $K$-panel, $\Omega:=E \cup -E$ and denote by $M(\Omega)(K^\text{sep})$ its fixator in $G(K^\text{sep}).$
Then $M(\Omega)=Z \langle V_\alpha, V_{-\alpha}\rangle$ for the $K$-root $\alpha$ with $E \subseteq \partial\alpha.$  

The group $M(\Omega)$ is a reductive algebraic group defined over $K$ of split semisimple rank 1, which can be seen by considering $\Ad_\Omega(M_\Omega)$. It follows that a rational root group $V_\alpha$ is isomorphic to a root group of a semisimple $K$-group (cf. \cite[12.5.4]{Remy}).  \\%

Let $N(K)$ denote the stabilizer of $A_K$ in $G(K).$ Then $W^\natural:=N(K)/Z(K)$ is called the {\bf relative Weyl group}.
It can be shown that $W^\natural$ is in fact a Coxeter group with generating set $S^\natural$ whose set of roots is in bijection with the half-apartments of $A_K$, see below.\\

R\'emy proved the following important and difficult theorem (\cite[Theorem 12.4.3]{Remy}).
\begin{theorem}
 Let $G$ be an almost split Kac--Moody $K$-group which is obtained by Galois descent. Let $(A_K,F,-F)$ be a rational standardisation. Then the group of rational points $G(K)$ is endowed with a twin root datum $(Z_A(K),(V_\alpha)_{\alpha \in \Phi(W^\natural,S^\natural)})$.
\end{theorem}

{\bf Geometric realization of the associated twin building.} 
It can be checked \cite[12.4.4]{Remy} that the set of $\Gamma$-fixed points in $\Delta(G(E))$ gives a geometric realization of the twin building associated to $G(K)$ in the sense that adjacency and opposition can be checked by looking at the fixed points in $\Delta_{cone}(G(E))$.\\

Just like in the finite-dimensional case (cf. \cite[Chapter 42]{TitsWeiss}), we have the following fact:
\begin{proposition} Let $G(K)$ be a quasi-split Kac--Moody group obtained via Galois descent. Then the derived group of the anisotropic kernel $Z$ is trivial, i.e. $Z(K)$ is abelian.
\end{proposition}

\begin{proof}
By definition, the Galois group $\Gamma$ stabilizes two opposite Borel groups of $G(E)$, where $E$ is a splitting field of $G$. Without loss of generality, these can be assumed to be the standard Borel groups $B_+,B_-.$ By the explicit description of the generic subspace $A_K$ it follows that $A_K$ is entirely contained in the cone of $C_+$ and $C_-.$ So any element $g \in G(E)$ which fixes $A_K$ will stabilize both $B_+$ and $B_-$, from which it follows that $g \in T(E).$ Thus $Z(K)\leq T(E)$, which is abelian. 
\end{proof}

\subsection{Facts about isotropic reductive algebraic groups}
Let $G=G(k)$ be an almost split Kac--Moody group obtained via Galois descent. Let $\Omega$ a balanced subset of $\Delta_{cone}$ and let $M:=\Fix_{G(k)}(\Omega)$. Then $\Ad_\Omega(M)$ can be identified with the $k$-points of an algebraic group defined over $k$, and $M$ itself is a central extension of this group.\\
(The fact that $\Ad M$ is defined over $k$ is implied by the axioms that the adjoint representation be Galois equivariant and that $G(k)$ is obtained by Galois descent; this is one of the main motivations of introducing these two axioms.) \\
This is why we recall here some facts about 
$k$-rational points of algebraic groups. For the following facts see \cite{BorelTitsRG} or \cite[Section 1.2]{Deodhar}, where a convenient summary of the results we need is given.
Let $k$ be a field, $\bar{k}$ an algebraic closure of $k$ and $G$ a connected reductive linear algebraic group defined over $k$. For our purposes, we can assume that $G$ comes with a fixed embedding, i.e. $G$ is a Zariski-closed subgroup of some $\GL_n(\bar{k})$. \\
Let $S\leq G$ be a maximal $k$-split torus and $X^*(S)$ its character group. Suppose that $G$ is isotropic over $k$, i.e. $S$ is non-trivial.

Let $\Phi \subseteq X^*(S)$ be the corresponding $k$-root system of $G$ with respect to $S$, i.e. the set of weights of $S$ acting on $\lieg:=\Lie G$ via the adjoint representation.\\
For $\alpha \in \Phi$, let $\lieg_\alpha \subseteq \lieg$ denote the corresponding root space, i.e.
$$\lieg_\alpha=\{X \in \lieg: \Ad s (X)=\alpha(s)\cdot X \; \forall \; s \in S\}.$$
Let $\lieu_\alpha:=\sum_{k>0} \lieg_{k\alpha}$ and let $U_\alpha$ be the connected unipotent subgroup of $G$ with $\Lie U_\alpha=\lieu_\alpha.$ In fact, the only positive multiples of $\alpha$ which could possibly belong to $\Phi$ are $\alpha$ and $2\alpha$. 
The group $U_\alpha$ then is split over $k$, cf. \cite[Cor. 3.18]{BorelTitsRG} and normalized by the centralizer $Z:=C_G(S)$ of $S$ in $G$. \\
If $\alpha \in \Phi$ is such that $2\alpha \not\in \Phi$, then $U_2:=U_\alpha$ is $k$-isomorphic to a vectorspace ${\bf G}_a{}^n$.  
If $\alpha \in \Phi$ is such that $2\alpha \in \Phi$, then $U_1:=U_\alpha/U_{2\alpha}$ again is isomorphic over $k$ to a vectorspace. \\
In both cases, under this identification the action of $S$ on $U_1$ resp. $U_2$ is given via the homothety induced by $\alpha.$ This means that for $s \in S(k)$ and $u \in U_1(k)$ or $u \in U_2(k)$, we have
$$ s\cdot u \cdot s^{-1}=\alpha(s)\cdot u.$$

\subsection{Further properties of Kac--Moody $K$-groups.} For the rest of this paper, any almost split Kac--Moody groups is understood to be obtained via Galois descent.\\
We briefly recall the discussion of reductive $k$-subgroups of $G$ as given in \cite[12.5.2]{Remy} to make the interplay of the maximal split torus and the relative root groups of an almost split Kac--Moody group explicit. \\

Let $k$ be a field and let $G=G(k)$ be an almost split Kac--Moody $k$-group which splits over a separable extension $E\subseteq k^{sep}.$ Let $(A_k,F,-F)$ denote a rational standardisation. \\

By definition, $F$ and $-F$ are two minimal Galois-stable opposite spherical facets of the twin building associated to $G(E)$. The stabilizer of $\Omega:=F \cup -F$ in $G(E)$ can then be identified with a Levi factor $L^J(E):=T\langle U_\alpha : \alpha \in \Phi(W_J) \rangle$ where $J\subseteq S$ is spherical. From the defining relations of the constructive Tits functor, it follows that $L^J(E)$ is abstractly isomorphic to the $E$-points of a connected reductive group split over $E$. Since $L^J$ is invariant under the $\Gamma$-action, it follows that $L^J$ is defined over $k$. Write $Z$ for the algebraic group $L^J$ endowed with this $k$-structure.
So $Z(E)\cong L^J(E)$, while $Z(k)$ of course is in general very different from $L^J(k).$\\
For $\Omega$ as above, $\Ad_\Omega(Z)$ is a connected semisimple algebraic group defined over $k$ which is anisotropic over $k$. It follows that there exists a unique maximal $k$-split torus $T_d$ contained in $Z$. The torus $T_d$ is central in $Z$ and can be identified with a maximal $k$-split subtorus of $T$. \\

More generally, let $x \in A_k$ be a $k$-facet. Then for $\Omega:=x \cup -x$, the fixator of $\Omega$ in $G(k)$ can be identified with the $k$-rational points of some Levi factor $L^{J'}$ of $G(E)$, where the $k$-structure on $L^{J'}$ again is given by the $\Gamma$-action. (We dealt above with the case when $x$ is a $k$-chamber.)\\
The point here is that fixators of opposite points of a twin apartment carry an intrinsic structure of (the $k$-points of) an algebraic group. For bounded subgroups in general, though, one has to pass to the adjoint representation.\\ 

We combine this discussion with the review of rational points of algebraic groups in the previous subsection to sum up the interplay between the maximal split torus $T_d(k)$ and the root groups $V_a(k).$\\
Let ${\bf G}_a$ denote the algebraic group with ${\bf G}_a(k)=(k,+).$
For a group $G$, let $\mathscr Z(G)$ denote the center of $G$ (which should not be confused with the anisotropic kernel $Z$ of an almost split Kac-Moody group).

\begin{proposition} \label{TorusActsOnRootGroups} 
Let $k$ be an infinite field and let $G$ be an almost split Kac--Moody group obtained by Galois descent. Let $Z$ be the anisotropic kernel of $G$, $T_d \leq Z$ a maximal $k$-split torus
and $W_k$ the Weyl group of $G(k)$ with $S_k$ its set of canonical generators. Let $\Phi_k=\Phi(W_k,S_k)$ denote the set of $k$-roots and $(V_\alpha(k))_{\alpha \in \Phi_k}$ the set of root groups of $G(k)$ relative to $T_d(k)$. \\
Let $\Pi_k$ denote the set of simple roots of $\Phi_k.$
 
\begin{enumerate}
\item $Z$ is a connected reductive algebraic group defined over $k$. 
The torus $T_d$ is a maximal $k$-split torus of $Z$ which is central in $Z$; the derived group of $Z$ is anisotropic over $k$.
\item Let $J \subseteq S_k$ be such that $(W_k)_J$ is finite. Then 
$L^J:=Z\langle V_\alpha: \alpha \in \Phi((W_k)_J)\rangle$
is a connected reductive algebraic $K$-group, in which $T_d$ is a maximal $k$-split torus.
$L^J$ has split-semisimple rank $|J|$. 
\item Let $\alpha \in \Delta_k.$ Then $X_\alpha:=Z\langle V_\alpha, V_{-\alpha}\rangle$ is a connected reductive algebraic $k$-group 
of split-semisimple rank 1. 
$V_\alpha$ is a root group in $X_\alpha$ normalized by $Z$. There are two possibilities:
\begin{enumerate}
\item [i)]
$V_\alpha$ is abelian and is $k$-isomorphic to ${\bf G}_a^n$ for $n:=\dim V_\alpha.$ In this case, $V_\alpha$ is normalized by $Z$, and $T_d$ acts on $V_\alpha$ via a character $\alpha$. This means there is some $\alpha \in X^*(T_d)$ defined over $k$ such that $tut^{-1}=\alpha(t)\cdot u$ for $t \in T_d$ and $u \in V_\alpha.$
\item [(ii)]
$V_\alpha$ is metabelian. Then $\mathscr Z(V_\alpha)$ is $k$-isomorphic to ${\bf G}_a^n$, where $n:=\dim \mathscr Z(V_\alpha)$, and 
$V_\alpha/\mathscr Z(V_\alpha)$ is $k$-isomorphic to ${\bf G}_a^{n'}$, where $n':=\dim V_\alpha - \dim \mathscr Z(V_\alpha)$.\\
The anisotropic kernel $Z$ normalizes both $V_\alpha$ and $\mathscr Z(V_\alpha).$ There is a character $\alpha \in X^*(T_d)$ defined over $k$ such that $T_d$ acts on $\mathscr Z(V_\alpha)$ via $2\alpha$ and on $V_\alpha/\mathscr Z(V_\alpha)$ via $\alpha.$
\end{enumerate}
\item Let $u \in \mathscr Z(V_\alpha(k)) \backslash\{1\}$ and $s_\alpha:=m(u)=u'uu''$ the associated $\mu$-map. Then $s_\alpha$ normalizes $T_d(k).$
\item Let $\alpha \in \Phi_k.$ If $t \in T_d$ centralizes some $u \in V_\alpha\backslash\{1\}$, then $t^2$ already centralizes $V_\alpha.$
\item If $\alpha, \beta \in \Phi_k, \alpha \neq \pm \beta$ are such that $o(s_\alpha s_\beta)<\infty$, then there is an element $t \in T_d(k)$ such that $t$ centralizes $V_\alpha$ but not $V_\beta.$
\end{enumerate}
\end{proposition}

\begin{proof}
Part a) is clear by the above discussion; similarly, as $L^J$ is the fixator of two opposite points $x$, $-x$, for b) it is sufficient to check the statement about the semisimple rank of $L^J$, which follows from the fact that $\Ad_{x \cup -x}(L^J)$ is a semisimple group in which the $(V_\beta: \beta \in \Phi(W_k)_J)$ form a system of root groups in the algebraic sense.\\
Part c) follows from b) and the discussion of rational points of semisimple algebraic groups in the previous subsection.\\
For part d), note that by c) $X_\alpha$ is a reductive group with $T_d$ a maximal split torus. Then the Zariski closure of $\{s u s^{-1}: s \in T_d\}$ is a one-dimensional subgroup of $V_\alpha$, and so is part of a maximal split reductive subgroup $F \leq X_\alpha$ which contains $T_d$, as follows from the Borel--Tits theorem (see Theorem 5.1). As $m(u)$, computed in $F$, leaves $T_d$ invariant, so must $m(u)$, as computed in $X_\alpha.$\\
Part e follows from part c) by noting that if $V_\alpha$ is abelian, then necessarily $\alpha(t)=1$ (and so already $t$ must centralize $V_\alpha$). In case $V_\alpha$ is metabelian, if $u \in \mathscr Z(V_\alpha)$, then $2\alpha(t)=\alpha(t^2)=1$ (so $t^2$ centralizes $V_\alpha$), while 
if $u \not\in \mathscr Z(V_\alpha)$, then $\alpha(t)=1$, so $t$ already centralizes $V_\alpha.$\\
For part f) it follows from the assumption that $V_\alpha, V_\beta$ are contained in some Levi factor $L^J$ with $|J|=2.$ Since the characters associated to $\alpha$ and $\beta$ are not proportional, 
$C_{T_d}(V_\alpha)=\ker \alpha$ does not contain $C_{T_d}(V_\beta)=\ker \beta.$ As $T_d(k)$ is Zariski dense in $T_d$,
the claim follows. 
\end{proof}
\subsection{Restriction of scalars for Kac--Moody groups}
We give a class of examples of quasi-split Kac--Moody groups obtained by the classical process of restriction of scalars, cf. \cite[Section 2.1.2]{PlatonovRapinchuk}. These examples show that
an abstract isomorphism $\psi: G_1 \to G_2$ of two almost split Kac--Moody groups does not in general preserve the full parameter set $(\DD_i,E_i|K_i,\rho_i)$ attached 
to these groups.

\begin{proposition}\label{RestrictionScalars}
Let $k$ be a field and let $E|k$ be a finite Galois extension. Let $G$ be a split Kac--Moody group. Then there is a quasi-split Kac--Moody group $G'$ such that $G(E)$ is isomorphic to $G'(k).$
\end{proposition}

\begin{proof}
Let $\Gamma:=\Gal(E|k), n:=|\Gamma|$ and let $G_0$ be the direct product of $n$ copies of $G$, indexed by the elements of $\Gamma.$ Define an action of $\Gamma$ on $G_0(E)$ by setting
$$ \gamma \cdot (g_{\sigma_1}, \ldots, g_{\sigma_n}):=(g_{\gamma\sigma_1},\ldots, g_{\gamma\sigma_n}).$$

Let $G'(k)$ denote the fixed point set of $\Gamma$ acting on $G_0(E)$. Then $G'(k)$ is precisely the diagonal subgroup of $G_0(E)$, which is isomorphic to $G(E)$. \\
It remains to be checked that this $\Gamma$-action is the $*$-action induced by a $\Gamma$-action on the Dynkin diagram of $G_0$, which allows to apply the results of \ref{RemyConstruction}. This is immediate, though, as the Dynkin diagram of $G_0$ is the disjoint union of $n$ copies of the Dynkin diagram of $G$, and $\Gamma$ permutes these copies. 
\end{proof}

\begin{remark} Let $E|k$ be a finite Galois extension and let $G$ be a connected almost simple $k$-group which is split over $k$. Then the group $G'(k)\cong G(E)$ provided by Proposition \ref{RestrictionScalars} is the group classically obtained by restriction of scalars. The isomorphism $\varphi\colon G(E) \to G'(k)$ is not covered by Borel--Tits's theory \cite{BorelTits} since $G'(k)$ is not absolutely almost simple. Indeed, in this theory one restricts to absolutely almost simple groups for precisely this reason. 
\end{remark}

\section{Maximal split subgroups}
\subsection{Split subgroups of groups of Kac--Moody type}
An almost split Kac--Moody group $G(k)$ obtained via Galois descent is by definition a subgroup of a split Kac--Moody group $\GG_\DD(E).$
On the other hand, we show in this section that $G(k)$ possesses a maximal split subgroup $F(k)$ of Kac--Moody type, i.e.\ a subgroup
endowed with a twin root datum which is locally split and intersects each root group $V_\alpha(k)$ of $G(k)$ non-trivially.  

\begin{example}
Let $k$ be a field and let $E|k$ be a separable extension of degree 2. Let $h: E^3 \to E$
be a Hermitian form of Witt index 1 with associated unitary group $\SU_3$, which can be thought of as an algebraic group defined over $k$. 
Then $\SU_3(k[t,t^{-1}])$ is an almost split Kac--Moody group obtained from the split Kac--Moody group $\SL_3(k[t,t^{-1}])$ via Galois descent, cf. Example \ref{SU3example}.\\
 On the other hand, there is an inclusion $\SL_2(k[t,t^{-1}]) \leq \SU_3(k[t,t^{-1}])$ as
 for the associated root groups $(V_\alpha(k)_{\alpha \in \Phi(W,S)})$
 of $\SU_3([k[t,t^{-1}])$ it follows that $\langle \mathscr Z(V_\alpha(k)): \alpha \in \Phi\rangle \cong  \SL_2(k[t,t^{-1}])$. \\ 
The twin building associated to $\SU_3(\Fq_q[t,t^{-1}])$ is a semi-regular twin tree with valencies $(1+q,1+q^3)$ in which the twin building associated to 
$\SL_2(\Fq_q[t,t^{-1}])$, a regular twin tree with valency $1+q$, embeds.
\begin{figure}[ht]
\centering
\includegraphics[scale=.7]{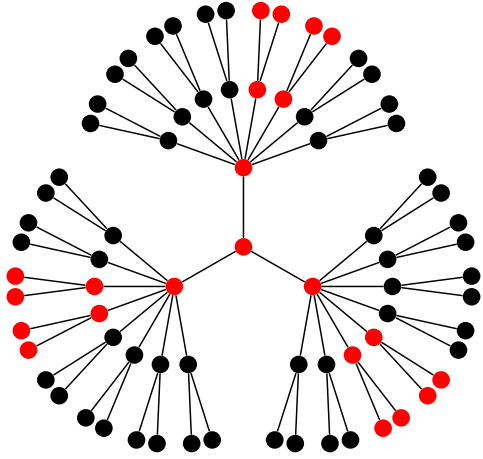}
\caption{One half of the twin tree associated to the inclusion $\SL_2(\Fq_2[t,t^{-1}]) \leq \SU_3(\Fq_2[t,t^{-1}])$} 
\end{figure}
\end{example}

\begin{example}\label{chainorthogonalgroups}
Let $k$ be a field of characteristic $\neq 2$, $n\geq 2$ and let $q=\langle a_1, \ldots, a_n\rangle$ be a quadratic form of Witt index 1 over $k$. We may assume that $\langle a_1, a_2 \rangle = \langle 1,-1 \rangle$ and that $\langle a_3, \ldots, a_n \rangle$ is anisotropic. Let $G:=\SO(q)$ denote the associated special orthogonal group.\\
For $r=2, \ldots, n$, let $q_r:=\langle a_1, \ldots, a_r \rangle$ denote the truncated quadratic form and let $G_{q_r}:=\SO(q_r)$ denote the associated special orthogonal group. Note in passing that $T:=G_{q_2}(k) \cong k^\times$ and $G_{q_3}(k) \cong \PGL_2(k)$ -- this follows from the fact that $G_{q_3}$ is a split three-dimensional semisimple group, so it is either isomorphic to $\SL_2$ or $\PGL_2$, and these groups can be distinguished by the torus action on the root groups.\\
The point of this example is that there is a chain of reductive $k$-groups 
$$T=G_{q_2} \leq G_{q_3} \leq \ldots \leq G_{q_n}=G$$
which share the same maximal torus $T=G_{q_2}$ of $G$. While $G_{q_3}$ is split and contains the maximal split torus $T$, clearly it is not the only subgroup of $G_{q_n}$ with this property -- any $G_i':=\SO(\langle 1,-1,a_i\rangle)$ for some $i \in \{4, \ldots, n\}$ has the same property, and $G_{q_3},G_i'$ are not conjugate over $k$ if
$a_3a_i^{-1} \not\in k^2.$
\end{example}

%

The following is a classical result by Borel-Tits (\cite[Theorem 7.2]{BorelTitsRG}).

\begin{theorem}\label{BorelTitsSplit}
Let $G$ be a connected reductive $k$-group. Let $S\leq G$ be a maximal $k$-split torus, $\Phi=\Phi(S,G)$ 
the system of $k$-roots of $G$ and $\Phi'\subseteq \Phi$ the set of non-multipliable roots. Let
$\Pi$ be a set of simple roots of $\Phi'$ and for each $\alpha \in \Pi$ let $E_\alpha\leq U_\alpha$ be a $k$-subgroup
which is normalized by $S$ and is $k$-isomorphic to {\bf G}$_a$. Then there is a unique connected $k$-split reductive $k$-subgroup $F$ which contains $S\cdot \langle E_\alpha: \alpha \in \Pi \rangle.$
\end{theorem}

We prove a generalization of this result for a group $G$ endowed with a 2-spherical root datum, which might be of independent interest as it provides ``many'' sub-twin buildings of the twin building associated to $G$. In our context, it will be used to construct a regular diagonalizable subgroup $H \leq G$ which is mapped under any isomorphism $\varphi\colon G \to G'$ again to a regular diagonalizable subgroup (cf. section 5).\\

In a first step we define the necessary ingredients of a locally split subgroup and then go on to prove that these ingredients ``integrate'' to a locally split group of Kac--Moody type.\index{Coxeter group!2-spherical}\\
Recall that a Coxeter group $W=\langle (s_i)_{i \in I} : (s_is_j)^{m_{ij }}=1 \rangle$ is said to be {\bf 2-spherical} if $m_{ij}<\infty$ for all $i,j \in I.$

For elements $x,y \in G$ write ${}^y x:=yxy^{-1}.$ For a group $G$, let $G^*:=G \backslash \{1\}.$

%
%
%

\begin{definition}
Let $W=\langle (s_i)_{i \in I} : (s_is_j)^{m_{ij}}=1 \rangle$ be a 2-spherical Coxeter group and let $G$ be a group endowed with a twin root datum $(H,(U_\alpha)_{\alpha \in \Phi(W,S)})$. \\
Let $\Pi=\{\alpha_1,\ldots, \alpha_n\}$ denote the set of positive simple roots.\\
Let $T_d$ be a subgroup of $H$ and for each $\alpha \in \Pi$ let $E_\alpha \leq U_\alpha$ be a non-trivial subgroup. For $\alpha \in \Pi$, let $s_\alpha:=m(v)$ for some $v \in E_\alpha^*$. 

Then $(T_d,(E_\alpha)_{\alpha \in \Delta})$ is called a {\bf basis for a root subdatum} if the following conditions are satisfied: \index{basis for a root subdatum}

\begin{enumerate}[(RSD 1)]
\item For all $i,j$, $(s_{\alpha_i}s_{\alpha_j})^{m_{ij}} \in T_d$.
\item For all $r,t \in  E_\alpha^*$, $m(r)m(t)^{-1} \in T_d$. 
\item For all $\alpha \in \Pi$, $E_\alpha$ is normalized by $T_d$ and each $s_\alpha$ normalizes $T_d$.
\item For $v \in E_\alpha^*$ there exist $v_1, v_2 \in E_\alpha^*$ such that $m(v)(:=v'vv'')={}^{s_{\alpha}}v_1 \cdot v \cdot {}^{s_{\alpha}}v_2.$
\item If $X \leq U_{(\alpha,\beta]}$ is a subgroup normalized by $T_d$
and $x=u_1u_2 \in X$ with $u_1 \in U_{(\alpha,\beta)},u_2 \in U_\beta$, then $u_1,u_2 \in X$.

\end{enumerate}
\end{definition}

As the name suggests, a basis for a root subdatum gives rise to a subgroup which has a twin root datum.

\begin{theorem}\label{subbuilding} Let $(W,S)$ be a 2-spherical Coxeter group, let $\Phi=\Phi(W,S)$ denote the set of its roots and let $\Pi$ be the set of simple roots.

Let $G$ be a group endowed with a twin root datum $(H,(U_\alpha)_{\alpha \in \Phi(W,S)})$. Let
$(T_d,(E_\alpha)_{\alpha \in \Pi})$ be a basis for a root subdatum.\\
Let $M:=T_d \langle s_\alpha: \alpha \in \Pi\rangle$,  $V:=\langle E_\alpha: \alpha \in \Pi \rangle$ and $F:=\langle M,V\rangle.$ Set $F_\gamma:=F \cap U_\gamma$ for $\gamma \in \Phi.$ \\
Then $(T_d,(F_\gamma)_{\gamma \in \Phi})$ is a twin root datum of type $(W,S)$ for $F$.

\end{theorem}

The proof, which will be given after a couple of preparatory lemmas, 
is very much inspired by \cite[Proof of Theorem 7.2]{BorelTitsRG}.

\begin{lemma} \label{posnegcommute}
 Let $G$ be a group endowed with a twin root datum $(H,(U_\alpha)_{\alpha \in \Phi(W,S)})$. Let $\alpha, \beta$ be two distinct positive simple roots. Then $U_{-\alpha}$ commutes with $U_\beta.$
\end{lemma}

\begin{proof}
The set $\Psi:=\{-\alpha,\beta\}$ is a prenilpotent set of roots since $s_\beta \Psi \subseteq \Phi^-$ and $s_\alpha \Psi \subseteq \Phi^+$. The open root interval $(-\alpha,\beta)$ is empty: Any positive root in $[-\alpha,\beta]$ must be mapped to a negative root by $s_\beta$ and hence coincides with $\beta$, while any negative root in $[-\alpha,\beta]$ must be mapped to a positive one by $s_\alpha$ and hence coincides with $-\alpha.$ 
By the commutator axiom, $[U_{-\alpha},U_{\beta}] \leq U_{(-\alpha,\beta)}=1.$
\end{proof}

We first analyze the structure of $V$.\\
Let $E_{-\alpha}:={}^{s_{\alpha}}E_{\alpha}.$ Then $E_{-\alpha}$ is independent from the choice of $v \in E_\alpha^*$ in the definition of $s_{\alpha}=m(v)$ as for $v,v' \in E_\alpha^*$,
$m(v)$ and $m(v')$ differ by an element of $T_d$ by (RSD 2), and $T_d$ normalizes $E_\alpha$ by (RSD 3).\\
For $\alpha, \beta \in \Pi$ let $E_{(\alpha,\beta)}:=[E_\alpha,E_\beta]$ denote the commutator subgroup. Then $E_{(\alpha,\beta)}$ is normalized by $E_\beta$, since for $a \in E_\alpha,b,c \in E_\beta$, $$c[a,b]c^{-1}=caba^{-1}b^{-1}c^{-1}=[c,a][a,cb].$$ 
Let $E_{(\alpha,\beta]}:=E_{(\alpha,\beta)}\cdot E_\beta.$
\begin{lemma} \label{StructureOfV}

Let $\alpha, \beta \in \Pi$ be two distinct positive roots.
\begin{enumerate}
\item
$E_{(\alpha,\beta]}=\langle u_\alpha u_\beta u_\alpha^{-1}: u_\alpha \in E_\alpha, u_\beta \in E_\beta, u_\alpha \neq 1 \rangle.$
\item 
 $E_{(\alpha,\beta]}$ is normalized by $s_\alpha.$
\item
Let $E_\alpha':=\langle E_{(\alpha,\gamma]}: \gamma \in \Pi, \gamma \neq \alpha\rangle.$ Then $V=E_\alpha \ltimes E_\alpha'.$
\end{enumerate}
\end{lemma}

\begin{proof}
\begin{enumerate}
\item 
Let $X:=\langle u_\alpha u_\beta u_\alpha^{-1}: u_\alpha \in E_\alpha, u_\beta \in E_\beta, u_\alpha \neq 1 \rangle.$ Then $X \leq E_{(\alpha,\beta]}$ is clear. Conversely, since $E_\alpha$ and $E_\beta$ are normalized by $T_d$, so is $X.$ For $u_\alpha \in E_\alpha^*, u_\beta \in E_\beta$, note that $u_\alpha u_\beta u_\alpha^{-1}=[u_\alpha, u_\beta]u_\beta \in U_{(\alpha,\beta)}U_\beta.$ By (RSD 5) it follows that $u_\beta$ and $[u_\alpha,u_\beta]$ are contained in $X$, from which the claim follows.

\item 
By a), it suffices to show that $s_\alpha (u_\alpha u_\beta u_\alpha^{-1}) s_\alpha^{-1} \in E_{(\alpha,\beta]}$, where $u_\alpha \neq 1.$ Write $s_\alpha=u_1u_2 u_{\alpha}^{-1}$ for some $u_1 \in E_\alpha, u_2 \in E_{-\alpha}$ -- this is legitimate as $u_\alpha \neq 1$ and $s_\alpha$ is defined only up to elements of $T_d$. \\
Then 
$$ s_\alpha (u_\alpha u_\beta u_\alpha^{-1}) s_\alpha^{-1} = u_1 u_2 u_\beta u_2^{-1} u_1^{-1} = u_1 u_2 u_1^{-1}$$
since $u_2 \in E_{-\alpha}$ commutes with $u_\beta$ by Lemma \ref{posnegcommute}, from which the claim follows. 
\item 
It is clear that $E_\alpha$ and $E_\alpha'$ are subgroups of $V$ which generate $V$. From (i) it is immediate that $E_\alpha$ normalizes $E_\alpha'.$ Let $v \in E_\alpha \cap E_\alpha'$. Then by b)
$${}^{s_\alpha}v \in {}^{s_\alpha}E_\alpha \cap {}^{s_\alpha}E_\alpha'=E_{-\alpha} \cap E_\alpha' \leq U_- \cap U_+=1.$$
\end{enumerate}
\end{proof}

\begin{lemma} \label{StructureOfW}
\begin{enumerate} 
\item There is a canonical isomorphism $\pi: M/T_d \to W.$
\item Let $\alpha \in \Delta$ and $w \in W$ be such that $w\alpha$ is positive. Then ${}^wE_\alpha \leq V$.
\end{enumerate}
\end{lemma}

\begin{proof}
\begin{enumerate}
\item Note that $T_d$ is a normal subgroup of $M$ by (RSD 3); by (RSD 1) it follows that $M/T_d \cong W.$ 
\item Since $T_d$ normalizes $E_\alpha$, ${}^wE_\alpha$ is well-defined. If $l(w)=0$, there is nothing to prove, so suppose $l(w)\geq 1.$ Since $w\alpha>0$, we can write $w=s_\beta w'$, where $\beta$ is a simple root distinct from $\alpha$ and $w'$ is such that $w'\alpha>0.$ By induction, ${}^{w'} E_\alpha \leq V=E_\beta \ltimes E_\beta'.$ Since ${}^{w'} E_\alpha \leq U_{w' \alpha}$ and $w' \alpha \neq \beta$, it follows that ${}^{w'} E_\alpha \leq E_\beta'.$ 
Then ${}^w E_\alpha\leq {}^{s_\beta}E_\beta'=E_\beta'$.
\end{enumerate}
\end{proof}

%
%

%
%
%
\index{Bruhat decomposition}
The next step consists of exhibiting a Bruhat decomposition for $F$.
\begin{lemma}\label{BruhatForF}
The group $F$ can be written as $F=VMV=\cup_{w \in W} VwV.$ 
\end{lemma}

\begin{proof}
The set $V\cdot M\cdot V$ contains $V$ and $M$, is stable under inversion and closed under multiplication by elements in $V$ or $T_d$ from the right or left. To show that it coincides with $F$, it thus suffices to check that it is closed under multiplication from the right by $s_\alpha, \alpha \in \Delta.$\\

{\it First step.} For $\alpha \in \Delta$, $E_{-\alpha} \subseteq T_dE_\alpha \cup T_dE_\alpha s_\alpha E_\alpha.$ \\
Indeed, $1 \in T_dE_\alpha$ while by definition each $v \in E_{-\alpha}^*$ has the form $v={}^{s_\alpha} v_0$ for some $v_0 \in E_\alpha$. By (RSD 4), there are $v_1, v_2 \in E_\alpha$ such that $m(v_0)={}^{s_{\alpha}}v_1 v_0 {}^{s_\alpha}v_2.$ Then ${}^{s_\alpha^{-1}}s_\alpha=v_1vv_2,$ i.e.\ $v \in E_{\alpha}s_\alpha E_{\alpha}$, from which the claim follows.\\

{\it Second step.}
Since $T_d$ normalizes $V$, we can write $VMV=\cup _{w \in W} VwV$ unambiguously. We will show that $VwV s_\alpha \subseteq Vws_\alpha V \cup VwV$, from which the claim will follow.\\

If $l(w)=0$, i.e.\ $w=1$, then by the first step and Lemma \ref{StructureOfV},
$$Vs_\alpha=E_\alpha s_\alpha {}^{s_\alpha}E_\alpha'\subseteq E_\alpha s_\alpha V.$$
Suppose $l(w)\geq 1$ and the claim is proven for all $w'$ with $l(w')<l(w).$
Two cases can occur:\\
(1) $l(ws_\alpha)>l(w).$ This is the case if and only if $w\alpha>0$. Then $wE_\alpha w^{-1} \leq V$ by Lemma \ref{StructureOfW} and we calculate
$$VwVs_\alpha=VwE_\alpha s_\alpha E_\alpha'=VE_\alpha w s_\alpha E_\alpha' \subseteq Vws_\alpha V.$$
(2) $l(ws_\alpha)<l(w)$, i.e.\ $w\alpha<0.$ Then we can write $w=w's_{\alpha}$ with $l(w')=l(w)-1 \geq 0.$ We calculate
$$VwVs_\alpha=Vw's_\alpha E_\alpha s_\alpha E_\alpha'=Vw'E_{-\alpha}E_\alpha'\subseteq Vw'V \cup Vw'Vs_\alpha V=Vw'V \cup Vw's_\alpha V.$$
Here the last equality follows because $w'\alpha>0$, which allows us to apply the first case.
\end{proof}

%

We can turn to the proof of Theorem 4.5.

\begin{proof} 
For $\gamma \in \Phi \backslash \Pi$ and $w \in W, \alpha \in \Pi$ such that $w\alpha=\gamma$ choose some lift $\tilde{w} \in M$ of $w$ and set $E_\gamma:=\tilde{w}E_\alpha\tilde{w}^{-1}.$ Then for each $\gamma\in \Phi$, $E_\gamma \subseteq F_\gamma.$ Assume for the moment that equality holds (in particular, $E_\gamma$ will then not depend on the choice of $\alpha$ and $\tilde{w}$). \\
Then clearly for each $\gamma \in \Phi$, $F_\gamma$ is nontrivial and normalized by $T_d$ by (RSD 3). By (RSD 4), $s_\alpha \in \langle E_{-\alpha},E_\alpha\rangle$, from which it follows that $F$ is generated by $T_d$ and $\langle E_\alpha, E_{-\alpha}: \alpha \in \Delta \rangle$, i.e.\ (TRD 1) holds.
Set $V_-:=\langle F_\gamma: \gamma<0\rangle.$ Then $V_-\cap V \leq U_-\cap U_+=1$ and therefore (TRD 4) is satisfied. Similarly, (TRD 2) holds by the definition of $F_\gamma$ and the corresponding property for $G$. \\
Axiom (TRD 3) holds for $F_\gamma$, $\gamma \in \Delta$ by (RSD 2) and (RSD 4). 


It remains to prove that $F_\gamma=E_\gamma$ for $\gamma \in \Phi$, in particular $F_\alpha=E_\alpha$ for $\alpha \in \Pi$ which is not clear {\it a priori}. \\


%

{\it First step. If $\gamma \in \Pi,$ then $F \cap U_\gamma=E_\gamma$}.\\
By the Bruhat decomposition $F=VMV$ it follows that 
$F \cap U_\gamma=V \cap U_\gamma$. Since $V=E_\gamma \ltimes E_\gamma'$ it follows that
$$ s_\gamma (F \cap U_\gamma) s_\gamma^{-1}=s_\gamma V s_\gamma^{-1} \cap U_{-\gamma} \\
 = E_{-\gamma} E_\gamma' \cap U_{-\gamma}=E_{-\gamma}\cdot (E_\gamma' \cap U_{-\gamma})=E_{-\gamma},$$
 from which it follows that $F \cap U_\gamma=E_\gamma.$\\
 
{\it Second step. If $\delta \in \Phi\backslash \Pi$ is arbitrary, then $F \cap U_\delta=E_\delta$.}\\
Suppose first that $\delta \in \Phi^+$.
Let $w=\tilde{w} \in M,\alpha \in \Pi$ as in the definition of $E_\delta.$ Then
$$w(F\cap U_\delta)w^{-1}=w(V \cap U_\delta)w^{-1}=wVw^{-1} \cap (U_+ \cap U_\alpha) \subseteq V \cap U_\alpha=E_\alpha.$$
By definition, $\tilde{w}E_\alpha\tilde{w}^{-1} \subseteq F_\delta$, and we have just shown the reverse inclusion, i.e.\ $F_\delta=E_\delta$.\\
Clearly the same reasoning works when $\delta \in \Phi^-$, which finishes the proof of the theorem.
\end{proof}

\begin{remark} The statement of Lemma \ref{BruhatForF} that $F=\cup_{w\in W} V w V$ can be thought of as the fact that $F$ is a graded subgroup of $G$.
This means that whenever $f=b_1wb_2$ with $b_1, b_2 \in B$ and $w \in W$ is the Bruhat decomposition of an element $f \in F$, then $b_1,b_2$ and $w$ can actually be chosen to be elements of $F$.
\end{remark}

\begin{remark}
Let $G$ be a group endowed with a 2-spherical twin root datum \\
$(H,(U_\alpha)_{\alpha \in \Phi(W,S)}).$
Then $(H,(U_\alpha)_{\alpha \in \Delta})$ meets conditions (RSD 1)-(RSD 4), but not necessarily (RSD 5). Indeed, if (RSD 5) is met, it follows from the proof of the preceding theorem that $U_+=\langle U_\alpha: \alpha \in \Delta\rangle.$ This is satisfied for isotropic reductive $k$-groups with $|k|\geq 4$, but fails e.g. for $G_2(\Fq_2).$ 
\end{remark}

\begin{remark} A geometric interpretation of the theorem is as follows: 
Let $\Delta$ be the twin building associated to $G$, $\mA$ the twin apartment determined by $H$ and $C_+,C_-$ the two opposite chambers corresponding to $B_+,B_-.$ On each panel $F_\alpha$ of $C_+$, fix chambers according to the action of $E_\alpha$ on $F_\alpha.$ Condition (RSD 4) ensures that these form a sub-Moufang set. The remaining conditions  are the necessary compatibility conditions which ensure that these chambers give rise to a sub twin building with $\mA$ as a twin apartment. \\
In particular, the twin building $\Delta(F)$ associated to $F$ embeds in $\Delta(G)$ as a closed convex subcomplex. 
\end{remark}

\subsection{The case of almost split Kac--Moody groups}
We apply Theorem \ref{subbuilding} to almost split Kac--Moody groups.

Let $G$ be a group endowed with a twin root datum $(H,(U_\alpha)_{\alpha \in \Phi(W,S)}).$
\index{twin! root datum!locally split}
Then the twin root datum is said to be {\bf locally split} (over a family of fields $(k_\alpha)_{\alpha \in \Phi})$ if $H$ is abelian and for each $\alpha \in \Phi$, 
$\langle U_\alpha, U_{-\alpha} \rangle$ is isomorphic to either $\SL_2(k_\alpha)$ or
$\PSL_2(k_\alpha).$

\begin{theorem}\label{maximalsplitsubgroup}
Let $k$ be an infinite field and let $G(k)$ be a 2-spherical almost split Kac--Moody group obtained by Galois descent. Let $(Z(k),(V_\alpha)_{\alpha \in \Phi(W,S)})$ denote its canonical twin root datum and let $T_d(k)\leq Z(k)$ be a maximal $k$-split torus. \\
For each simple root 
$\alpha \in \Pi$ let $E_\alpha \leq V_\alpha$ be a subgroup isomorphic to $(k,+)$ which is normalized by $T_d(k)$. Then there is a subgroup
$F(k)\leq G(k)$ which contains $T_d(k) \langle E_\alpha: \alpha \in \Pi \rangle$ and which is endowed with a locally split twin root datum. 
\end{theorem}
\begin{proof}
Since $G$ is assumed to be 2-spherical, for each pair of simple roots $\{\alpha,\beta\} \subseteq \Delta$ the group $X_{\alpha\beta}:=Z(k) \langle V_{\pm \alpha}, V_{\pm \beta} \rangle$ can be identified with the $k$-points of a reductive algebraic $k$-group of relative rank 2. By Theorem \ref{BorelTitsSplit}, there is a split subgroup $Y_{\alpha\beta}\leq X_{\alpha\beta}$ which contains $T_d(k)$ and $E_\alpha, E_\beta$. Now $Y_{\alpha\beta}$ is endowed with a spherical twin root datum, the properties of which imply that the axioms (RSD 1) - (RSD 4) of a root subdatum are satisfied, since these need to be checked only for rank 2 subgroups.\\
Since $k$ is infinite, $T_d(k)$ is Zariski dense in $T_d$. For a subgroup $X \leq V_{(\alpha,\beta]}$ normalized by $T_d(k)$ it follows that $\overline{X}$ is normalized by $T_d$. By \cite[Proposition 3.11]{BorelTitsRG} it follows that (RSD 5) is satisfied as well. \\

Theorem \ref{maximalsplitsubgroup} gives the existence of $F$, and from the fact that the group $Y_{\alpha\beta}$ is a split reductive group it is immediate that the twin root datum for $F$ is locally split.
\end{proof}

\begin{definition} Let $k$ be an infinite field and let $G(k)$ be a 2-spherical almost split Kac--Moody group obtained by Galois descent. 
Any group $F$ obtained from $G(k)$ in this way is called a {\bf maximal split subgroup} of $G$.
\end{definition}

\begin{remark}
It is always poosible to find subgroups $E_\alpha$ as required in Theorem 4.13: just let $E_\alpha$ be a one-dimensional $k$-subspace of $\mathscr Z(V_\alpha)$. Then Proposition \ref{TorusActsOnRootGroups}(iii) b) shows that $E_\alpha$ is normalized by $T_d(k).$ 
This proves Theorem 1.2.
\end{remark}

In particular, any almost split 2-spherical Kac--Moody group is ''sandwiched'' between two split Kac--Moody groups: For a splitting field $E$ of $G$, one has
$$ F(k) \leq G(k) \leq G(E).$$

Here the Coxeter type of $F(k)$ is the same as the Coxeter type of $G(k)$, while the type of $G(k)$ equals the type of $G(E)$ if and only if $G$ is already split split over $k$.  
%


\begin{remark} We used Theorem \ref{subbuilding} to produce a locally split subgroup. The theorem is more general, though, 
as arbitrary sub-Moufang sets are allowed. In particular, we recover Example \ref{chainorthogonalgroups}. 
\end{remark}
\begin{remark}
Another example of a basis for a root subdatum $(T_d,(E_\alpha)_{\alpha \in \Delta})$ as required in Theorem \ref{subbuilding} comes from subfields: If $k \subseteq K$ and $\GG_\DD$ is a constructive Tits functor, then take $T_d:=T(k)$ and $E_\alpha:=U_\alpha(k)$ inside $\GG_\DD(K)$. 
Of course, the theorem can be applied more than once, i.e.\ pass first to a locally split subgroup and then to $k$-rational points.\\
\end{remark}

%

Finally, just as Example \ref{chainorthogonalgroups} suggests, there is a chain condition on groups containing a maximal split subgroup.
\begin{proposition}
Let $k$ be an infinite field, char $k \neq 2$ and let $G(k)$ be an almost split Kac--Moody group over $k$ obtained by Galois descent. Let $F \leq G$ be a maximal split subgroup. Let 
$$F=H_1 \leq H_2 \leq \ldots$$
be a chain of subgroups which are obtained by integrating a root subdatum and such that $H_i=H_i^\dagger$, i.e.\ $H_i$ is generated by its root groups. Then the chain eventually becomes stationary. \\
More precisely, the length of a strictly increasing chain is bounded by ${\displaystyle \sum_{\Pi \in \Delta} \dim_k V_\alpha}$.
\end{proposition}

\begin{proof}
For each simple root $\alpha \in \Pi$ with corresponding root group $V_\alpha(k) \leq G(k)$  let $H_{i,\alpha}:=H_i \cap V_\alpha.$ 
Then $H_{i,\alpha} \cap \mathscr Z(V_\alpha)$ is a $k$-sub-vectorspace since it is invariant
under $T_d(k)$, similarly for $(H_{i,\alpha}\cdot \mathscr Z(V_\alpha))/\mathscr Z(V_\alpha).$ \\
Recall from Proposition \ref{TorusActsOnRootGroups} that $V_\alpha(k)$ is an extension of two finite-dimensional $k$-vector spaces. 
This implies that $(H_{i,\alpha})$ eventually becomes stationary. Since the $H_i$ are supposed to be generated by their root groups, the first claim follows.\\
Since in a strictly increasing chain of subgroups, in each step there is some $\alpha \in \Pi$ such that $H_{i,\alpha}$ is strictly contained in $H_{i+1,\alpha}$, the second claim follows.
\end{proof}
%
%
%
%
%
%
\section{The isomorphism problem}
In this chapter we prove that every abstract isomorphism of two 2-spherical almost split Kac--Moody groups over fields of characteristic 0 is standard in the sense that it induces an isomorphism of the associated canonical twin root data. 
\subsection{Preparatory lemmas}
\begin{lemma} \label{FGSubgroupDense} Let $k$ be an infinite field and let $T$ be a $k$-split torus.
Let $S\leq T(k)$ be such that $T(k)/S$ is finitely generated. Then $S$ is Zariski dense in $T$.
\end{lemma}
\begin{proof}
Since $T$ is split over $k$ and $k$ is infinite, $T(k)$ is Zariski dense in $T$. Assume that $\bar{S}\neq T$. Then $\bar{S}$ is defined over $k$ and so is $\bar{S}^0$, which is a $k$-split subtorus of $T$ by \cite[Corollary 1.9 b)]{BorelTitsRG}.
It follows that $\dim \bar{S}^0<\dim T$. Passing to the rational points, we find that $X:=T(k)/(S\cap \bar{S}^0(k))$ contains a copy of $k^\times$, which is not finitely generated. As $S\cap \bar{S}^0(k)$ has finite index in $S$, $X$ is finitely generated. This is a contradiction since 
any subgroup of a finitely generated abelian group is finitely generated.
\end{proof}

\begin{proposition}\label{group_in_rootgroups}
Let $G$ be a group endowed with a twin root datum $(H,(U_\alpha)_{\alpha\in \Phi(W,S)})$. Let $L$ be a subgroup of $G$ such that for each $l \in L \backslash\{1\}$ there is a root $\beta_l \in \Phi(W,S)$ such that $l \in U_{\beta_l}$. Then there is some $\beta \in \Phi(W,S)$ such that $L \leq U_\beta.$
\end{proposition}

\begin{proof}
Note first that $\beta_x=\beta_{x^{-1}}$ as $U_{\beta_x}$ is a subgroup and $\beta_x$ is uniquely determined as $U_\alpha \cap U_\beta=1$ for distinct roots $\alpha \neq \beta.$\\
Assume for a contradiction that there are $x,y \in L\backslash\{1\}$ such that $\beta_x \neq \beta_y$, in particular $xy\neq 1.$ Let $\alpha:=\beta_x, \beta:=\beta_y$ and $\gamma:=\beta_{xy}.$\\
This implies that $U_\alpha U_\beta \cap U_\gamma \neq 1.$
Moreover, for any permutation $\pi$ of $\{\alpha, \beta, \gamma\}$, it follows that $U_{\pi (\alpha)} U_{\pi (\beta)} \cap U_{\pi (\gamma)} \neq \{1\}$: If $u_\alpha u_\beta=u_\gamma$, then $u_\beta^{-1}u_\alpha^{-1}=u_\gamma^{-1}$, $u_\beta u_\gamma^{-1}=u_\alpha^{-1}$, and the permutations $(\alpha\, \beta),(\alpha\, \gamma\, \beta)$ generate $\Sym(\{\alpha, \beta,\gamma\}).$ This implies that if two of the three roots coincide, then all roots coincide. So we can suppose that all three roots are distinct. \\
If two of the three roots are positive and the remaining one is negative (or vice versa), we can assume that $\alpha>0,\beta>0$ and $\gamma<0$ since the statement to be proved is invariant under permutations of the roots. But this is a contradiction as $U_+ \cap U_-=\{1\}.$
If all three roots have the same sign, say $\alpha, \beta,\gamma \in \Phi^+$, then choose some $w \in W$ such that $w\gamma=\delta$ is a positive simple root. If $w\alpha$ or $w\beta$ is negative, this is a contradiction by the case just discussed. If $w\alpha, w\beta,w\gamma$ are all positive, then $s_\delta w\alpha>0, s_\delta w\beta >0, s_\delta w_\gamma <0$, which is again a contradiction by the case just discussed.
\end{proof}


\begin{proposition}\label{existenceofsl2}
Let $k$ be a field of characteristic $0$ and let $G$ be a connected reductive algebraic group defined over $k$. Let $g \in G(k) \backslash\{1\}$ be a nontrivial unipotent element. Then there exists a morphism $\varphi\colon \SL_2 \to G$ defined over $k$ such that 
$\varphi(u)=g$ for some unipotent element $u \in \SL_2(k).$
\end{proposition}
\begin{proof}
Let $U:=\overline{\langle u \rangle}$. As $k$ is of characteristic 0, $U$ is a one-dimensional unipotent group which is defined over $k$ since $u \in G(k).$ This implies that $U$ is $k$-isomorphic to ${\bf G}_a$. Let $\mathfrak{u}:=\Lie U$. By the Jacobson-Morozov lemma
(usually stated for {\it semisimple} Lie algebras over a field of characteristic 0, but holding in fact for arbitrary completely reducible subalgebras $\lieg \leq \mathfrak{gl}(V)$, see the original paper \cite[Theorem 3]{Jacobson1951}), there is a three-dimensional Lie subalgebra $x$ which is $k$-isomorphic to $\mathfrak{sl}_2$ and contains $\mathfrak{u}.$ As $\chark k=0$, any perfect Lie subalgebra is the Lie algebra of a closed subgroup $X$ (\cite[Corollary 7.9]{Borel}).  
This translates into the fact that there is a closed subgroup $X\leq G$ defined over $k$ which is $k$-isomorphic to either $\SL_2$ or $\PGL_2$. This implies the claim. 
\end{proof}

Let $G$ be a connected reductive group defined over $k$ which splits over $k$. Let $T$ be a maximal torus and $U$ a unipotent group which is normalized by $T$. Then $\langle T,U \rangle$ is contained in a Borel group $B$, so there is an ordering on the set of roots $\Phi(T,G)$ of $T$ in $G$ such that $U\leq U_+.$ It is then a classical fact (cf. \cite[p.65 l.7]{BorelTits}) that $U$ is generated by the root groups $U_\alpha$ relative to $T$ which are contained in $U$. 
 
We need an analogue of this theorem in case that $G$ is not necessarily split over $k$.

\begin{proposition} \label{positiveordering}
Let $k$ be an infinite field. 
Let $G$ be a connected reductive $k$-group which is $k$-isotropic and let $S\leq G$ be a maximal $k$-split torus. Let $U\leq G$ be a unipotent subgroup defined over $k$ which is normalized by $S$. Then $U$ is contained in $\langle U_\alpha: \alpha >0 \rangle$ for some ordering $>$ of the set of roots $\Phi(S,G)$ of $S$ in $G$.
\end{proposition}
\begin{proof}
Let $P$ be a minimal parabolic subgroup defined over $k$ which contains $U$ and $S$. Then $P$ has a Levi decomposition $P=Z(S)P_u$, where 
$Z(S)$ is the centralizer of $S$ and $P_u$ is the unipotent radical of $P$. Since $S$ is maximal $k$-split, $Z(S)(k)$ does not contain any unipotent elements. This implies that $U(k) \leq P_u(k)$ and since $U(k)$ is dense in $U$, it follows that $U \leq P_u$, which implies the claim. \end{proof}


\index{diagonalizable subgroup!}
\index{diagonalizable subgroup!regular}
Recall that a subgroup $S$ of an almost split Kac--Moody $k$-group $G(k)$ is called {\bf diagonalizable} (over $k$) if 
there is $g \in G(k)$ such that $gSg^{-1}\leq T_d(k)$, where $T_d$ is the standard maximal $k$-split torus of $G(k)$.\\
Furthermore, a diagonalizable subgroup $S$ is called {\bf regular} if the fixed point set of the $S$-action on the associated twin building consists of a single twin apartment.\\

Among all diagonalizable subgroups of $G$, regular subgroups can be characterized purely group-theoretically. The following characterization can be found in \cite[Proposition 5.13] {PECAbstract} for split Kac--Moody groups. We generalize this to almost split Kac--Moody groups, where care has to be taken of the anisotropic kernel.

\begin{lemma}\label{characterizationregular}
Let $k$ be a field of characteristic 0 and let $G$ be an almost split Kac--Moody $k$-group. Let $S\leq G(k)$ be a diagonalizable subgroup. Then $S$ is regular if and only if $S$ does not centralize a subgroup $X \leq G(k)$ isomorphic to either $\SL_2(\QQ)$ or $\PSL_2(\QQ).$
\end{lemma}

\begin{proof}
Without loss of generality we may assume that $S$ is contained in the standard maximal $k$-split torus $T_d(k)$. Suppose first that $S$ is regular and centralizes $X$. As $X$ has a fixed point in both $\Delta_+$ and $\Delta_-$ and is normalized by $S$, these points can be assumed to lie in the standard twin apartment $A_k.$ As $X$ normalizes $S$, it must stabilize $A_k$.
Hence there is a homomorphism $\Psi\colon X \to W=\Stab_G(A_k)/\Fix_G(A_k).$ $\Psi(X)$ then is a finite group as it is a subgroup of a point stabilizer, so a finite index subgroup $X'\leq X$ is contained in the anisotropic kernel $\Fix_G(A_k)=Z(k).$ Then $X'=X$ as $\PSL_2(\QQ)$ is simple and $\SL_2(\QQ)$ does not have a proper finite index subgroup either. (Indeed, since $U_+(\QQ)$ and $U_-(\QQ)$ are divisible, any finite index subgroup $N \leq \SL_2(\QQ)$ contains $U_+(\QQ)$ and $U_-(\QQ)$, hence is equal to $\SL_2(\QQ)$.) \\
Postcomposing with the adjoint representation $\Ad_\Omega$, where $(A_k,F,-F)$ is a rational standardisation and $\Omega:=\{F,-F\}$, there is a homomorphism $$X \to \Ad_\Omega(Z(k))$$ which induces a representation of $\SL_2(\QQ)$. This representation is rational and defined over $k$ by \cite[Lemma 5.9]{PECAbstract}. Since the target group is anisotropic over $k$ and therefore does not contain $k$-rational unipotent elements, this homomorphism must be trivial. Then $X \leq \ker \Ad_\Omega$, which is a contradiction since the latter group is abelian.\\ 
Conversely, suppose that $S$ fixes a point $x \not \in A_k$, without loss of generality suppose that $x \in \Delta_+.$ Then there is a panel $E$ of $A_k$ and a chamber $C_1 \not\subseteq A_k$ which has $E$ as a panel and is fixed by $S$. Indeed, let $\GG=(C_0, C_1, \ldots, C_n)$ be a gallery such that $C_0 \in A_k, C_n$ contains $x$ and $\GG$ is of minimal length among all such galleries. Then $n \geq 1$ since $x \not\in A_k$, and $C_1$ is fixed by $S$ since the $S$-action is type-preserving.\\
Let $E=C_0 \cap C_1$ and  let $\alpha$ be the corresponding root of $A_k$ determined by $C_0$ and $E$. The root group $V_\alpha \leq G(k)$ parametrizes the chambers which have $E$ as a panel and which are different from $C_0$. 
Since $S$ fixes $A_k$ and $C_1 \not\subseteq A_k$, there are three chambers of the $E$-panel fixed by $S$. This means that there is some non-trivial $v \in V_\alpha$ centralized by $S$. If $v \in V_\alpha \backslash \mathscr Z(V_\alpha)$, this implies that $S$ centralizes the entire group $V_\alpha$, if $v \in \mathscr Z(V_\alpha)$, this implies at least that $S$ centralizes $\mathscr Z(V_\alpha)$ (recall that the action of the split torus is via a character on both  $V_\alpha/\mathscr Z(V_\alpha)$ and $\mathscr Z(V_\alpha)$).\\
In either case, $S$ centralizes $\mathscr Z(V_\alpha)$ and also $\mathscr Z(V_{-\alpha}).$ Hence $S$ centralizes the group $T_d(k) \langle \mathscr Z(V_\alpha),\mathscr Z(V_{-\alpha})\rangle$. But this group contains a split semisimple group of rank 1 by Theorem \ref{BorelTitsSplit}, i.e. either $\SL_2(k)$ or $\PGL_2(k)$. In both cases the claim follows.  \end{proof}


\begin{lemma}\label{AbelianFG}
Let $(W, S)$ be a finitely generated Coxeter group and let $A\leq W$ be
a solvable subgroup. Then $A$ is finitely generated.
\end{lemma}
\begin{proof}
Let $X$ be the CAT(0) realization of $W$. Since $W$ is finitely generated, $X$ is
finite-dimensional by construction and $W$ acts on $X$ properly and cocompactly. The
conclusion follows now from \cite[p. 439 Theorem I.1 (3),(4)]{BridsonHaefliger}.
\end{proof}

\subsection{Proof of the main theorem}~\\
{\bf Setting.}
Let $k,k'$ be two fields of characteristic 0 and let $G,G'$ be two 2-spherical almost split Kac--Moody groups over $k,k'$, respectively. Let $G(k)$, $G'(k')$ denote their rational points and suppose that $\varphi\colon G(k) \to G'(k')$ is an abstract isomorphism.\\
Let 

\begin{itemize}
\item $Z(k)\leq G(k), Z'(k') \leq G'(k')$ denote the respective anisotropic kernels of $G,G'$.
\item $T_d(k) \leq Z(k), T_d'(k') \leq Z'(k')$ denote the respective maximal split tori. 
\item  $(W,S)$, $(W',S')$ denote the respective Weyl groups, and $\Phi, \Phi'$ the sets of roots with simple roots $\Pi, \Pi'$.
\item $(U_\alpha)_{\alpha \in \Phi}$ denote the rational root groups of $G$, and $(V_\beta)_{\beta \in \Phi'}$ the rational root groups of $G'.$
\item $\Delta$, $\Delta'$ denote the twin buildings associated to $G$ and $G'$. 
\item $\mA$ and $\mA'$ denote the standard twin apartments fixed by $Z(k),Z'(k')$ respectively.
\end{itemize}

{\bf Strategy of proof.} The proof strategy can be outlined as follows:\\
{\scshape Step 1.} Since $G(k)$ is assumed to be 2-spherical, $G(k)$ contains a maximal split subgroup $F(k)$ containing $T_d(k).$ A generalization of arguments from \cite{PECAbstract} can be used to exhibit a subgroup $S(\QQ) \leq T_d(k)$ with the property that $S(\QQ)$ fixes precisely $\mA$ and $\varphi(S(\QQ))$ fixes precisely a twin apartment $\mA''$ of $\Delta'.$ By postcomposing $\varphi$ with an inner automorphism if necessary, we assume
that $\mA''=\mA'.$ \\

{\scshape Step 2.} From the existence of $S(\QQ)$, which is in some sense a small subgroup of the split torus, we deduce the existence of two large subgroups $S_1 \leq T_d(k)$, $S_2 \leq T_d'(k')$ such that $\varphi(S_1) \leq Z'(k'),\varphi^{-1}(S_2) \leq Z(k).$ In particular, $\varphi(S_1)$ normalizes all root groups $V_\beta$ and $\varphi^{-1}(S_2)$ normalizes all root groups $U_\alpha.$ \\

{\scshape Step 3.} We now focus on a root group $U_\alpha$. Assume first that $U_\alpha$ is abelian (see Step 5 for the general case). 
Then for $u \in U_\alpha$, we show that $\varphi(u) \in L^J$ for some Levi factor $L^J$ of finite type, which depends a priori on $u$.\\
Using the groups $S_1$ and $S_2$ we show that $\varphi(u)$ actually is a unipotent element which is contained in a group $V_{\beta_1} \cdots V_{\beta_r} \leq L^J$.\\

{\scshape Step 4.} Now root groups in a spherical Levi factor can be distinguished by the torus action. Again using the groups $S_1$ and $S_2$, it follows that with the above notation, $r=1$, i.e. for each $u \in U_\alpha$ there is some $\beta_u \in \Phi'$ such that $\varphi(u) \in V_{\beta_u}.$\\
Since $U_\alpha$ is a group, it follows that $\varphi(U_\alpha)\leq V_{\beta(\alpha)}$ for some single $\beta(\alpha) \in \Phi'$. \\

{\scshape Step 5.} If $U_\alpha$ is not abelian, the analysis of steps 3 and 4 still applies to $\mathscr Z(U_\alpha)$. Let $u_1, \ldots, u_r$ be elements such that the canonical images of the $u_i$ are a $k$-basis for $U_\alpha/\mathscr Z(U_\alpha).$ Arguing as in steps 3 and 4 for the groups $k\cdot u_i$, together with the knowledge about $\varphi(\mathscr Z(U_\alpha))$ allows to conclude that also in this case $\varphi(U_\alpha)$ is contained in a single root group $V_{\beta(\alpha)}.$ \\

{\scshape Step 6.} By symmetry, each root group $V_\beta$ satisfies $\varphi^{-1}(V_\beta)\leq U_{\alpha(\beta)}$, so actually equality holds. 
This allows to conclude that $\varphi$ maps root groups to root groups and preserves the anisotropic kernel. \\

The following lemma is a key step in comparing the twin root data of $G$ and $G'$. 
 
\begin{lemma}
There exists a regular diagonalizable subgroup $S(\QQ)\leq T_d(k)$ such that $\varphi(S(\QQ))$ again is regular diagonalizable. 
\end{lemma}

\begin{proof} Fix a maximal split subgroup $F(k)$ of $G(k)$ and let $(T_d(k),(F_\alpha(k))_{\alpha \in \Phi})$ denote the associated twin root datum. 
Then each rank 2 subgroup $F_{\alpha\beta}(k):=T_d(k) \langle F_{\pm \alpha}(k), F_{\pm \beta}(k)\rangle$ coincides with the $k$-points of a split reductive group of semisimple rank 2. Since these groups are defined over $\ZZ$, it is possible to consider $F(\QQ)$, the $\QQ$-points of $F$. 
More precisely, for each root $\alpha \in \Phi$ let $f_\alpha\colon (k,+) \to F_\alpha(k)$ denote the corresponding isomorphism and $t\colon (k^\times)^n \to T_d(k)$ the canonical isomorphism. Then $F(\QQ):=t((\QQ^\times)^n) \cdot \langle f_\alpha(\QQ): \alpha \in \Phi \rangle.$\\
For each simple root $\alpha \in \Pi$ let $\psi_\alpha\colon \SL_2(\QQ) \to \langle F_\alpha(\QQ),F_{-\alpha}(\QQ) \rangle$ denote the canonical homomorphism. 
Let $D_\alpha:=\langle \psi_\alpha(\diag (x,x^{-1})): x \in \QQ^\times \rangle$ and let $S(\QQ):=\langle D_\alpha: \alpha \in \Delta\rangle.$\\
 
{\it Claim 1. $S(\QQ)$ is regular.} Note first that $S(\QQ)$ is invariant under the Weyl group. Indeed, it suffices to check that $s_\alpha(D_\beta) \leq S(\QQ)$ for two simple roots $\alpha,\beta$, and this can be verified in $F_{\alpha\beta}$ where it follows from the explicit description of the Weyl group action on the torus in a reductive group. Assume that $S(\QQ)$ is not regular. Then from the proof of Lemma \ref{characterizationregular} it follows that there is a root $\alpha$ such that $S(\QQ)$ centralizes $\mathscr Z(V_\alpha)$, i.e. the character $2\alpha$ vanishes on $S(\QQ)$. Write $\alpha=w\alpha_i$ for some $w \in W$ and a simple root $\alpha_i.$ Then 2$\alpha_i$ vanishes on $S(\QQ)$ by the Weyl group invariance of $S(\QQ)$, but this is a contradiction since $S(\QQ)$ contains $D_{\alpha_i}.$\\

{\it Claim 2. $\varphi(S(\QQ))$ is diagonalizable over $k'$}. Since $S(\QQ)$ is boundedly generated by $(D_\alpha)_{\alpha \in \Pi}$ and $\varphi(D_\alpha) \leq \varphi \circ \psi_\alpha(\SL_2(\QQ))$, it follows that $\varphi(S(\QQ))$ is bounded (see \cite[Section 2]{HainkeEmbedding}.  
Let $\Omega \subseteq \Delta'$ denote a balanced subset which is fixed by $\varphi(S(\QQ)).$ \\
Let $\overline{S(\QQ)}$ be the Zariski closure of $\Ad_\Omega(\varphi(S(\QQ))).$ As $S(\QQ)$ is commutative, so is $\overline{S(\QQ)}.$ Note that $\overline{S(\QQ)}$ is connected as it is generated by connected subgroups. \\

By \cite[3.1.1]{Springer}, $S:=\overline{S(\QQ)}$ is the direct product of its semisimple and its unipotent elements:
$S=S_s \times S_u$. Since the abstract representation $\rho:=\Ad_\Omega \circ \varphi \circ \psi_\alpha$ actually is rational, it follows that the image of each $S_\alpha(\QQ)$ consists of semisimple elements only, i.e. is contained in $S_s.$ \\
In particular, $S$ is a torus since it is connected and contains semisimple elements only. Clearly, $S$ is defined over $k'$. It 
remains to be checked that $S$ is split over $k'$. 
Let $g \in S(\QQ)$ be of infinite order. Since $g$ is contained in a $k$-split torus, the Zariski closure $S_g$ of $\langle g \rangle$ is again a $k$-split torus by \cite[Proposition 1.9 b)]{BorelTitsRG}. By induction, $S/S_g$ is a $k$-split torus, from which the result again follows by \cite[Proposition 1.9 b)]{BorelTitsRG}.
This implies the claim.\\ 


%

{\it Claim 3. $\varphi(S(\QQ))$ is regular diagonalizable.} This is a direct consequence of the group theoretic characterization of regular diagonalizable subgroups, Lemma \ref{characterizationregular}.
\end{proof}


\begin{remark} If $K$ is algebraically closed and $G$ is a split Kac--Moody group over $K$, it is even possible to exhibit \emph {finite} regular diagonalizable groups which are mapped to regular diagonalizable subgroups, see \cite{MR2180452}. 
Still in the split case over arbitrary fields, $T':=\ker (\alpha-\beta)$ for suitably chosen roots $\alpha,\beta$ is regular. In particular, the dimension of a regular diagonalizable subgroup can vary arbitrarily.
\end{remark}

\begin{remark} The assumption that $G(k),G'(k')$ be 2-spherical is essentially only used to produce a regular subgroup $S(\QQ) \leq G(k)$
which is again mapped to a regular diagonalizable subgroup in $G'(k').$
\end{remark}

Since maximal $k'$-split tori are conjugate under $G'(k')$ \cite[Theorem 12.5.3]{Remy}, there exists some $x \in G'(k')$ such that 
$(\intg x \circ \varphi) (S(\QQ)) \leq T_d'(k').$ Replacing $\varphi$ by $\intg x \circ \varphi$ if necessary, we assume from now on that $\varphi(S(\QQ)) \leq T_d'(k').$

\begin{proposition} There are subgroups $S_1 \leq T_d(k)$ and $S_2 \leq T_d'(k')$ with the property that $T_d(k)/S_1$, $T_d'(k')/S_2$ both are finitely generated and such that $\varphi(S_1) \leq Z'(k')$, $\varphi^{-1} (S_2) \leq Z(k)$. 
\end{proposition}
\begin{proof}
As $T_d(k)$ normalizes $S(\QQ)$, $T_d(k)$ acts via $\varphi$ on the fixed point set of $\varphi(S(\QQ))$, which is reduced to $\mA'$. Let $S_1$ denote the kernel of this action, then $\varphi(S_1) \leq Z'(k')$ by definition of the anisotropic kernel as the fixator of $\mA'$. As $\varphi(T_d(k))/\varphi(S_1)$ is an abelian subgroup of $W'$, it is finitely generated by Lemma \ref{AbelianFG}. \\
Similarly, as $T_d'(k')$ normalizes $\varphi(S(\QQ))$, $T_d'(k')$ acts via $\varphi^{-1}$ on the fixed point set of $S(\QQ)$, which is reduced to $\mA.$ Define $S_2$ as the kernel of this action, then $S_2$ is as required by similar arguments.
\end{proof}

The subgroups $S_1$ and $S_2$ should be thought of as ``large'' as they are Zariski dense in $T_d$ and $T_d'$, respectively, by Proposition \ref{FGSubgroupDense}. Moreover, $S_1$ and $\varphi^{-1}(S_2)$ both normalize each root group $U_\alpha\leq G$, while $\varphi(S_1)$ and $S_2$ both normalize each root group $V_\beta\leq G'$. \\

The next step consists of showing that for certain unipotent elements $u \in U_\alpha \backslash\{1\}$, $\varphi(u) \leq L^J$ for a Levi factor of spherical type containing $Z'(k').$
%
%


\index{pure element}
\begin{definition} Let $U_\alpha\leq G$ be a root group and let $\mathfrak{u}:=\Lie U_\alpha=\lieg_\alpha \oplus \lieg_{2\alpha}.$ For an element $u \in U_\alpha(k)$ let $\log u \in \mathfrak{u}$ denote the unique element such that $\exp (\log u)=u.$\\
Then $u \in U_\alpha$ is called {\bf pure} if $\log u \in \lieg_\alpha$ or $\log u \in \lieg_{2\alpha}.$
\end{definition}
Note that if $U_\alpha$ is abelian, each element $u \in U_\alpha \backslash\{1\}$ is pure.

\begin{lemma} \label{ExistenceOfSL2} 
Let $u \in U_\alpha(k) \backslash \{1\}$ be a pure element. Then there exists a homomorphism 
$\psi_u: \SL_2(\QQ) \to G(k)$ such that 
\begin{enumerate}
\item $\psi_u((\begin{smallmatrix} 1 & 1 \\ 0 & 1 \end{smallmatrix}))=u$
\item $\im \psi_u$ is normalized by $S(\QQ).$ 
\end{enumerate}
\end{lemma}

\begin{proof}
This follows from the proof of the Proposition \ref{existenceofsl2} or from Theorem \ref{BorelTitsSplit}. More precisely, 
since $u$ is pure, the subalgebra $k \log u$ is invariant under $\Ad T_d(k)$, i.e. there is a subgroup $Y_u \leq U_\alpha$ which contains $u$ and is isomorphic to ${\bf G}_a$. This isomorphism can be chosen to send $u$ to 1. By Theorem \ref{BorelTitsSplit}, $u$ is contained in a split group which contains $T_d \cdot Y_u$. Since $\QQ u$ is is invariant under $S(\QQ)$, the claim follows. \end{proof}

%

\begin{proposition}
Let $u \in U_\alpha \backslash\{1\}$ be a pure element. Then $\varphi(u)$ fixes two opposite points $x,y \in \mA'$, i.e. $\varphi(u) \in L^J$ for a Levi factor of finite type of $G'$ relative to $T_d'.$
\end{proposition}

\begin{proof}
Let $\psi_u: \SL_2(\QQ) \to G(k)$ be as in the previous lemma. 
Then $$\varphi \circ \psi_u: \SL_2(\QQ) \to G'(k')$$ is a homomorphism
whose images fixes two opposite points $x,y \in \Delta'$ by \cite[Proposition 2.7]{HainkeEmbedding}. As $\im \psi_u$ is normalized by $S(\QQ)$, both $x$ and $y$ must actually be contained in $\mA'$ by \cite[Lemma 2.1]{HainkeEmbedding}Lemma and the fact that $\varphi(S(\QQ))$ fixes only $\mA'.$ 
\end{proof}
It remains to prove that not only $\varphi(u) \in L^J$ but actually $\varphi(u) \in V_{\beta(u)}$ for some root $\beta(u)$ depending on $u$.\\

The following proposition uses the trick that a unipotent element $u$ is an element of the derived group of a solvable group $B_u$, a property which is clearly preserved by a group isomorphism. This idea goes back to \cite{BorelTits}.

\begin{proposition}\label{uInL}
Let $u \in U_\alpha \backslash\{1\}$ be pure and let $J \subseteq S'$ be such that 
$\varphi(u) \in L^J.$ Then 
$$ cl(u):=\overline{\langle c\varphi(u) c^{-1}: c \in S_2\rangle} \leq \overline{L^J}$$
is a unipotent group defined over $k'$ and normalized by $T_d'$.
\end{proposition}

\begin{proof}
Let $Y_u:=\langle \varphi^{-1}(c)u\varphi^{-1}(c^{-1}): c \in S_2 \rangle.$ Then $Y_u \leq U_\alpha$ since $\varphi^{-1}(S_2)$ normalizes $U_\alpha.$ Moreover $Y_u$ is contained in $Y_u':=\langle sY_us^{-1}: s \in S(\QQ)\rangle$. \\
By Proposition \ref{uInL} there is a subset $J \subseteq S'$ such that
$\varphi(u) \in L^J.$ Since $\varphi(S(\QQ))$ and $S_2$ are subgroups of $T_d'(k')$ it follows that $\varphi(Y_u')\leq L^J.$\\
Let $\QQ u:=\exp(\QQ \cdot \log u).$ Then $\QQ u$ is a group normalized by $S(\QQ)$ by Proposition \ref{ExistenceOfSL2}. The group $B_u:=S(\QQ)\cdot \QQ u$ is solvable and $u$ is contained in every finite index subgroup of the derived group of $B_u.$ Indeed, since $S(\QQ)$ acts on $\QQ \cdot u$ via a non-trivial character, for each $n \in \NN$ there is some $s \in S(\QQ)$ such that $\frac 1 n \cdot u \in \langle [s,u] \rangle.$ \\
As $B_u \leq Y_u', \varphi(B_u) \leq L^J$. Since $B_u$ is solvable, so is $\overline{\varphi(B_u)}$. By the Lie-Kolchin theorem, $\overline{\varphi(B_u)}$ has a finite index subgroup which is triagonalizable, and since $u$ is in the derived group, it follows that 
$\varphi(u)$ is unipotent. \\
Note that $S(\QQ)$ and $\varphi^{-1}({S_2})$ commute, i.e. $Y_u'$ is normalized by both $S(\QQ)$ and $\varphi^{-1}(S_2)$. Arguing similarly as for $u$, it follows that $\varphi(Y_u')$ is unipotent since it is (up to finite index) contained in the derived group of a solvable group.

%
Since $\varphi(Y_u') \leq G'(k')$, the Zariski closure $cl(u)$ is defined over $k'$, and $cl(u)$ again is unipotent (cf. \cite{Springer}). By definition, $cl(u)$ is normalized by $S_2$ and hence by the Zariski closure of $S_2$, which is $T_d'$ by Lemma \ref{FGSubgroupDense}.
\end{proof}

The following step is inspired by the proof of \cite[Proposition 23]{CapraceRemySimplicityLong}, which in turn is inspired by classical results. \\
We recall first the definition of a nibbling sequence of roots.

\index{nibbling sequence of roots}
\begin{definition} Let $(W,S)$ be a Coxeter group and let $\alpha_1\ldots, \alpha_n \in \Phi(W,S)$ be such that $\{\alpha_i,\alpha_j\}$ is prenilpotent for all $i,j \in \{1, \ldots, n\}.$ Then $(\alpha_1, \ldots, \alpha_n)$ is called a {\bf nibbling sequence} of roots if for all $i<j$, 
$(\alpha_i, \alpha_j) \subseteq \{\alpha_{i+1}, \ldots, \alpha_{j-1}\}.$
\end{definition}

\begin{proposition} \label{existsordering} Let $(W,S)$ be a spherical Coxeter group and let $\Psi \subseteq \Phi(W,S)$ be a nilpotent set of roots. Then the elements of $\Psi$ can be arranged to form a nibbling sequence of roots. 
\end{proposition}
\begin{proof}
See \cite[Section 9.1.2]{Remy}.
\end{proof}

\begin{theorem}\label{UnipotentInSingle}
Let $u \in U_\alpha \backslash\{1\}$ be pure and $J\subseteq S'$ spherical such that $\varphi(u)  \in L^J$. Then $\varphi(u) \in V_\beta$ for some $\beta \in \Phi'.$
\end{theorem}

\begin{proof}
By Proposition \ref{uInL} and Proposition \ref{positiveordering}, it follows that 
$$\psi(u) \in V_J^+:=\langle V_\beta: \beta \in \Phi(W'_J), \beta>0 \rangle$$ for a suitable ordering '$>$' of the roots of $\Phi(W'_J)$. 
Since $\Phi(W'_J)$ is finite it follows from Proposition \ref{existsordering} that there is an ordering on the positive roots $\beta_i$ such that $(\beta_1, \ldots, \beta_k)$ is a nibbling sequence.
Then 
$$\varphi(u)=v_{i_1}\cdots v_{i_r}$$
for certain $v_{i_j} \in V_{\beta_{i_j}}, v_{i_j} \neq 1$. \\
Assume for a contradiction that $r>1.$ \\
{\it Claim. In this case, there are indices $i \neq j$ and elements $u_i, u_j \in U_\alpha \backslash\{1\}$ such that 
$\varphi(u_i) \in V_{\beta_i}$ and $\varphi(u_j) \in V_{\beta_j}.$}

Since $W_J'$ is spherical, for any two roots $\beta_i, \beta_j \in \Phi(W_J')$ such that $\beta_i \neq \pm \beta_j$, there is an element $t_{ij} \in \langle V_\alpha: \alpha \in \Phi(W_J')\rangle \cap S_2$ such $t_{ij}$ centralizes $V_{\beta_i}$ but not $V_{\beta_j}.$
This follows from the fact that such an element exists in $T_d'(k')$ and the fact that $S_2$ is Zariski dense in $T_d'.$

Consider 
$v_1:=[t_{1,r},\varphi(u)]$ and $v_2:=[t_{r,1},\varphi(u)]$

Then $v_1,v_2 \in \varphi(U_\alpha)$ since $\varphi^{-1}(t_{ij})$ normalizes $U_\alpha.$ Furthermore, the support of $v_1$ contains $\beta_1$ but not $\beta_r$. Likewise, the support of $v_2$ does not contain $\beta_1$ but contains $\beta_r.$

By repeating the process for $v_1$ and $v_2$ inductively if necessary, the claim is proven.\\

Now take an element $s \in S_2$ of infinite order such that $s$ centralizes $V_{\beta_i}$
but not $V_{\beta_j}$ and such that $\varphi^{-1}(s) \in T_d(k).$ The existence of such an element can be proven by appealing to the $\QQ$-points of a split subgroup of $G'(k')$ and the fact that $\beta_i,\beta_j$ are roots in a {\it spherical} Coxeter group. 
Then $\varphi^{-1}(s)$ centralizes $u_i$, since $\varphi(u_i) \in V_{\beta_i}$, so $\varphi^{-1}(s^2)$ centralizes $U_\alpha$ since $\varphi^{-1}(s) \in T_d(k).$ But this is a contradiction, since $\varphi^{-1}(s^2)$ does not centralize $u_j$, since $\varphi(u_j) \in V_{\beta_j}$. 
\end{proof}

%

\begin{corollary} Let $U_\alpha$ be a root group. Then $\varphi(\mathscr Z(U_\alpha))\leq V_\beta$ for some $\beta \in \Phi'.$
\end{corollary}
\begin{proof}
By the preceding theorem, $\varphi(\mathscr Z(U_\alpha))$ is a group which satisfies the assumptions of Lemma \ref{group_in_rootgroups}, since each element $u \in \mathscr Z(U_\alpha) \backslash\{1\}$ is pure, so the conclusion follows.
\end{proof}

This corollary finishes the case where all root groups are abelian. Some more effort is required when there are metabelian root groups present. These technical problems are always present when one deals with metabelian root groups, see e.g. \cite{Deodhar} or \cite{BorelTits}.\\ 

The following lemma is inspired by the proof of \cite[Theorem 2.2]{MR2180452}.
\begin{lemma}
Let $u \in U_\alpha \backslash\{1\}$ be a pure element and let $\beta \in \Phi'$ be such that $\varphi(u) \in V_\beta$.
Then the elements $u',u'' \in U_{-\alpha}$ such that $m(u)=u'uu''$ satisfy $\varphi(u'), \varphi(u'') \in V_{-\beta}.$
\end{lemma}
\begin{proof}
From Lemma \ref{ExistenceOfSL2} it follows that $u'=u''$ and that $u'$ is pure. Let $\gamma \in \Phi'$ be such that $\varphi(u') \in V_\gamma$. It is clear that $\varphi(\QQ u) \leq V_\beta$ and that $\varphi(\QQ u')\leq V_\gamma$. This induces a homomorphism $\psi: \SL_2(\QQ) \to V_{\beta\gamma}:=\langle V_\beta, V_\gamma \rangle.$ Suppose that $\beta \neq -\gamma.$ If $\{\beta,\gamma\}$ is a prenilpotent set of roots, $V_{\beta\gamma}$ is nilpotent since each root group $V_\alpha$ is nilpotent, which is a contradiction since $\psi$ is nontrivial. If $\{\beta,\gamma\}$ is not prenilpotent, the free product $V_\beta * V_\gamma$ embeds in $G'(k')$, which is a contradiction since a conjugate of $u$ in $\SL_2(\QQ)$ commutes with $u'$, while this is not the case for $\varphi(u) \in V_\beta$ and $\varphi(u')\in V_\gamma.$ 
\end{proof}

%

\begin{proposition} Suppose that $U_\alpha$ is metabelian. Then $\varphi(U_\alpha)\leq V_\beta$ for some $\beta \in \Phi'$.
\end{proposition}

\begin{proof} Let $u_1, \ldots, u_r \in U_\alpha$ be pure such that $\log u_1, \ldots, \log u_r$ form a basis for $\lieg_\alpha.$
Let $U_i:=k\cdot u_i$ and let $U_0:=\mathscr Z(U_\alpha).$
 Let $\gamma_0, \ldots, \gamma_r \in \Phi'$ be such that $\varphi(U_i) \leq V_{\gamma_i}.$ These clearly exist, as each $U_i$ is a subgroup of $U_\alpha$ consisting of pure elements. \\
Suppose that there are $i,j$ such that $\gamma_i\neq\gamma_j$.
If $w:=s_{\gamma_i}s_{\gamma_j}$ has finite order, $\gamma_i$ and $\gamma_j$ are roots in a Levi factor $L^J.$ Then $U_{ij}:=\langle U_i,U_j\rangle$ is mapped to a unipotent subgroup of $L^J$ by arguments similar to those in the proof of Proposition \ref{uInL}.
Arguing as in the proof of Theorem \ref{UnipotentInSingle}, this yields a contradiction as then there would exist a torus element $t \in T_d'(k')$ such that $\varphi^{-1}(t)$ centralizes $U_i$ but not $U_j.$\\
It follows that $w$ has infinite order.
 Note that $\varphi(U_\alpha)$ is contained in the set 
$V':=V_{\gamma_1}\cdots V_{\gamma_r} \cdot V_{\gamma_0}$, in particular, 
$\varphi(U_\alpha)$ is bounded.\\

Let $m_i,m_j \in G'(k')$ be such that $m_i,m_j$ stabilize $\mA'$, 
act on it via $s_{\gamma_i},s_{\gamma_j}$ and such that $\varphi^{-1}(m_i),\varphi^{-1}(m_j)$ stabilize $\mA.$ These elements can be shown to exist via e.g. invoking a split subgroup of $G'(k')$. \\
From the previous proposition it follows that $\varphi^{-1}(m_i),\varphi^{-1}(m_j)$ map $U_\alpha$ to $U_{-\alpha}$.\\
For $t:=m_i m_j$ it follows that $\varphi^{-1}(t)$ normalizes $U_\alpha$. 

Then for each $r \in \ZZ$ there exists some $u_r \in U_\alpha$ such that $\varphi (u_r) \in V_{w^r \gamma_i}$. This is the desired contradiction, as this implies that $\varphi(U_\alpha)$ is unbounded.  \end{proof}

%
%
%
%
%
%

To sum up: For each $\alpha \in \Phi$ there is a root $i(\alpha) \in \Phi'$ such that $\varphi(U_\alpha)\leq V_{i(\alpha)}.$ \\
Arguing likewise for $\varphi^{-1}$ (note that the corresponding twin apartments $\mA,\mA'$ are already aligned in the right fashion) we find that for each $\beta \in \Phi'$ there is a $j(\beta) \in \Phi$ such that $\varphi^{-1}(V_\beta) \leq U_{j(\beta)}.$
From the inclusion
$$U_\alpha=\varphi^{-1}(\varphi(U_\alpha))\leq \varphi^{-1}(V_{i(\alpha)})
\leq U_{j(i(\alpha))}$$
and the fact that $U_\alpha\neq 1, U_\alpha \cap U_\beta=1$ for $\alpha \neq \beta$, it finally follows that $i$ and $j$ are inverse bijections and that equality holds all along.\index{standard isomorphism}\\

This discussion can be succinctly summed up by saying that any isomorphism $\varphi\colon G(k) \to G'(k')$ is {\bf standard}, cf. Definition \ref{standardisom}.
We have shown:
\begin{theorem} \label{maintheorem}
Let $G=G(k),G'=G'(k')$ be two 2-spherical almost split Kac--Moody groups over fields $k,k'$ of characteristic 0. Let $(Z(k),(U_\alpha)_{\alpha \in \Phi(W,S)})$ and $ (Z'(k'),(V_\beta)_{\beta \in \Phi(W',S')})$ denote the associated canonical twin root data. Suppose that $\varphi: G(k) \to G'(k')$ is an abstract isomorphism. Then $\varphi$ is standard.
\end{theorem}

\begin{proof} By the previous discussion, there exists $x \in G'(k')$ such that $\varphi':=\intg x \circ \varphi$ induces a bijection of the root groups. Note that $\intg x$ can be chosen to be trivial if, with the notation from above, $\varphi(S(\QQ))$ already fixes $\mA'.$
Since $$Z(k)=\bigcap_{\alpha \in \Phi(W,S)} N_{G(k)}(U_\alpha), \; Z'(k')=\bigcap_{\beta \in \Phi(W',S')} N_{G'(k')}(V_\beta)$$ (there is actually an equality, not just an inclusion, see \cite[Proposition 1.5.3]{Remy}), it follows that $\varphi'(Z(k))=Z'(k').$
\end{proof}

This proves Theorem 1.1 from the Introduction.


\begin{remark}
Since clearly $\mathscr Z(U_\alpha)$ is mapped to $\mathscr Z(V_\beta)$, it also follows that $\varphi$ induces an automorphism of the refined root datum as given by R\'emy \cite[Theorem 12.6.3]{Remy}.
\end{remark}


\begin{corollary}
Let $k,k'$ be two fields of characteristic 0 and let $G(k),G'(k')$ be two 2-spherical almost split Kac--Moody groups. Let $\varphi\colon G(k) \to G'(k')$ be an isomorphism rectified in such a fashion that $\varphi$ maps root groups with respect to $T_d(k)$ to root groups with respect to $T_d'(k').$ \\
Let $X_\alpha(k):=Z(k)\langle U_\alpha, U_{-\alpha}\rangle$ and 
$Y_\beta(k'):=Z'(k')\langle V_\beta,V_{-\beta}\rangle$ be two rank 1 groups such that $\varphi(X_\alpha)=Y_\beta.$

Suppose that the derived groups of both $X_\alpha$ and $Y_\beta$ are absolutely almost simple and that either $X_\alpha'$ is simply connected or that $Y_\beta'$ is adjoint. Then there is a field isomorphism $\sigma\colon k \to k'$, a rational map $r\colon X_\alpha' \to Y_\beta'$ and a map $c\colon X_\alpha'(k) \to \mathscr Z(Y_\beta'(k'))$ such that for $x \in X_\alpha'$, $\varphi(x)=c(x) \cdot (r \circ \sigma(x)).$
\end{corollary}

\begin{proof}
The assumption are made as to conform to the assumptions of Borel--Tits's classical theorem \cite[Theorem A]{BorelTits}, from which the claim follows.
\end{proof}

\bibliographystyle{plain}
\bibliography{LiteraturverzeichnisNeu}

\begin{thebibliography}{10}

\bibitem{AbramenkoBrown}
Peter Abramenko and Kenneth~S. Brown.
\newblock {\em Buildings}, volume 248 of {\em Graduate Texts in Mathematics}.
\newblock Springer, New York, 2008.

\bibitem{Borel}
Armand Borel.
\newblock {\em Linear algebraic groups}, volume 126 of {\em Graduate Texts in
  Mathematics}.
\newblock Springer-Verlag, New York, second edition, 1991.

\bibitem{BorelTitsRG}
Armand Borel and Jacques Tits.
\newblock Groupes r\'eductifs.
\newblock {\em Inst. Hautes \'Etudes Sci. Publ. Math.}, (27):55--150, 1965.

\bibitem{BorelTits}
Armand Borel and Jacques Tits.
\newblock Homomorphismes ``abstraits'' de groupes alg\'ebriques simples.
\newblock {\em Ann. of Math. (2)}, 97:499--571, 1973.

\bibitem{BridsonHaefliger}
Martin~R. Bridson and Andr{\'e} Haefliger.
\newblock {\em Metric spaces of non-positive curvature}, volume 319 of {\em
  Grundlehren der Mathematischen Wissenschaften}.
\newblock Springer, Berlin, 1999.

\bibitem{BuxGramlich}
K.U. Bux, R.~Gramlich, and S.~Witzel.
\newblock Finiteness properties of chevalley groups over a polynomial ring over
  a finite field.
\newblock {\em Arxiv preprint arXiv:0908.4531}, 2009.

\bibitem{PECAbstract}
Pierre-Emmanuel Caprace.
\newblock ``{A}bstract'' homomorphisms of split {K}ac--{M}oody groups.
\newblock {\em Mem. Amer. Math. Soc.}, 198(924):xvi+84, 2009.

\bibitem{MR2180452}
Pierre-Emmanuel Caprace and Bernhard M{\"u}hlherr.
\newblock Isomorphisms of {K}ac--{M}oody groups.
\newblock {\em Invent. Math.}, 161(2):361--388, 2005.

\bibitem{CapraceRemySimplicityLong}
Pierre-Emmanuel Caprace and Bertrand R{\'e}my.
\newblock Simplicity and superrigidity of twin building lattices.
\newblock {\em Invent. Math.}, 176(1):169--221, 2009.

\bibitem{Deodhar}
Vinay~V. Deodhar.
\newblock On central extensions of rational points of algebraic groups.
\newblock {\em Amer. J. Math.}, 100(2):303--386, 1978.

\bibitem{HainkeThesis}
G.~Hainke.
\newblock {\em The isomorphism problem for almost split Kac--Moody groups.}
\newblock PhD thesis, Westf\"alische Wilhelms-Universit\"at M\"unster, 2010.
\newblock Available online at
  \url{http://miami.uni-muenster.de/servlets/DerivateServlet/Derivate-5646/diss_hainke.pdf}.

\bibitem{HainkeEmbedding}
G.~Hainke.
\newblock Embeddings of algebraic groups in kac--moody groups.
\newblock {\em ArXiv e-prints}, August 2011.

\bibitem{Jacobson1951}
Nathan Jacobson.
\newblock Completely reducible lie algebras of linear transformations.
\newblock {\em Proceedings of the American Mathematical Society},
  2(1):105--113, 1951.

\bibitem{KrammerPhD}
Daan Krammer.
\newblock The conjugacy problem for {C}oxeter groups.
\newblock {\em Groups Geom. Dyn.}, 3(1):71--171, 2009.

\bibitem{PlatonovRapinchuk}
Vladimir Platonov and Andrei Rapinchuk.
\newblock {\em Algebraic groups and number theory}, volume 139 of {\em Pure and
  Applied Mathematics}.
\newblock Academic Press Inc., 1994.

\bibitem{remy2006topological}
B.~R{\'e}my and M.~Ronan.
\newblock Topological groups of kac--moody type, right-angled twinnings and
  their lattices.
\newblock {\em Commentarii mathematici helvetici}, 81(1):191--219, 2006.

\bibitem{Remy}
Bertrand R{\'e}my.
\newblock Groupes de {K}ac--{M}oody d\'eploy\'es et presque d\'eploy\'es.
\newblock {\em Ast\'erisque}, (277):viii+348, 2002.

\bibitem{RemySurvey}
Bertrand R{\'e}my.
\newblock Kac--{M}oody groups: split and relative theories. {L}attices.
\newblock In {\em Groups: topological, combinatorial and arithmetic aspects},
  volume 311 of {\em London Math. Soc. Lecture Note Ser.}, pages 487--541.
  Cambridge Univ. Press, Cambridge, 2004.

\bibitem{Springer}
T.~A. Springer.
\newblock {\em Linear algebraic groups}, volume~9 of {\em Progress in
  Mathematics}.
\newblock Birkh\"auser, Boston, second edition, 1998.

\bibitem{Tits3}
Jacques Tits.
\newblock Uniqueness and presentation of {K}ac--{M}oody groups over fields.
\newblock {\em J. Algebra}, 105(2):542--573, 1987.

\bibitem{Tits2}
Jacques Tits.
\newblock Groupes associ\'es aux alg\`ebres de {K}ac--{M}oody.
\newblock {\em Ast\'erisque}, (177-178):Exp.\ No.\ 700, 7--31, 1989.
\newblock S{\'e}minaire Bourbaki, Vol. 1988/89.

\bibitem{Tits1}
Jacques Tits.
\newblock Twin buildings and groups of {K}ac--{M}oody type.
\newblock In {\em Groups, combinatorics \& geometry ({D}urham, 1990)}, volume
  165 of {\em London Math. Soc. Lecture Note Ser.}, pages 249--286. Cambridge
  Univ. Press, Cambridge, 1992.

\bibitem{TitsWeiss}
Jacques Tits and Richard~M. Weiss.
\newblock {\em Moufang polygons}.
\newblock Springer Monographs in Mathematics. Springer-Verlag, Berlin, 2002.

\end{thebibliography}

\end{document}